\documentclass[11pt]{article}

\usepackage{times}
\usepackage{amsmath,amsfonts,amstext,amssymb,amsbsy,amsopn,amsthm,eucal}
\usepackage{txfonts}
\usepackage[T1]{fontenc}
\usepackage[utf8]{inputenc}
\usepackage{dsfont}
\usepackage{graphicx}
\usepackage{hyperref}
\usepackage{color}

\usepackage[active]{srcltx}

\hypersetup{
  pdftitle={Critical sets of elliptic equations},
  pdfauthor={Jeff Cheeger and Aaron Naber and Daniele Valtorta},
  pdfsubject={Minkowski and Hausdorff estimates for the Critical set of harmonic functions and solutions to elliptic equations},
  pdfkeywords={critical, singular, Hausdorff, Minkowski, Almgren, frequency function, quantitative, stratification, quantitative stratification, harmonic, elliptic, PDE, n-2, epsilon regularity},
  pdfpagelayout=SinglePage,
  pdfpagemode=UseOutlines}

\numberwithin{equation}{section}

\renewcommand{\v}{\vec v}

\newcommand{\Vol}{\text{Vol}}

\newcommand{\dR}{\mathds{R}}

\newcommand{\R}{\mathbb{R}}
\newcommand{\N}{\mathbb{N}}

\newcommand{\Cr}{\mathcal{C}}

\newcommand{\cC}{\mathcal{C}}

\newcommand{\cS}{\mathcal{S}}

\newcommand{\cW}{\mathcal{W}}
\newcommand{\cN}{\mathcal{N}}

\newcommand{\norm}[1]{\left\|#1\right\|}
\newcommand{\ps}[2]{\left\langle#1\middle\vert#2\right\rangle}

\newcommand{\ton}[1]{\left(#1\right)}
\newcommand{\qua}[1]{\left[#1\right]}
\newcommand{\cur}[1]{\left\{#1\right\}}
\newcommand{\abs}[1]{\left|#1\right|}

\newcommand{\dive}[1]{\operatorname{div}\ton{#1}}

\renewcommand{\L}{\mathcal{L}}

\begin{document}

\newtheorem{theorem}{Theorem}[section]
\newtheorem{ctheorem}[theorem]{Conjectural Theorem}

\newtheorem{proposition}[theorem]{Proposition}

\newtheorem{lemma}[theorem]{Lemma}
\newtheorem{clemma}[theorem]{Conjectural Lemma}

\newtheorem{corollary}[theorem]{Corollary}

\theoremstyle{definition}
\newtheorem{definition}[theorem]{Definition}

\theoremstyle{remark}
\newtheorem{remark}[theorem]{Remark}

\theoremstyle{remark}
\newtheorem{example}[theorem]{Example}

\theoremstyle{remark}
\newtheorem{note}[theorem]{Note}

\theoremstyle{definition}
\newtheorem{notation}[theorem]{Notation}

\theoremstyle{remark}
\newtheorem{question}[theorem]{Question}

\theoremstyle{remark}
\newtheorem{conjecture}[theorem]{Conjecture}

\title{Critical sets of elliptic equations}

\author{Jeff Cheeger\thanks{The author was partially supported by NSF grant DMS1005552} and Aaron Naber\thanks{The author was partially supported by NSF postdoctoral grant 0903137} and Daniele Valtorta}

\date{}
\maketitle
\begin{abstract}
 Given a solution $u$ to a linear homogeneous second order elliptic equation with Lipschitz
 coefficients, we introduce techniques for giving improved estimates of the critical set $\Cr(u)\equiv \{x:|\nabla u|(x)=0\}$, 
as well as the first estimates on the effective critical set $\cC_r(u)$, 
which roughly consists of points $x$ such that the gradient of $u$
 is large on $B_r(x)$ compared to the size of $u$.  The results are 
new even for harmonic functions on $\dR^n$.  Given such a $u$, 
the standard {\it first order} stratification $\{\cS^k\}$ of $u$ separates
 points $x$ based on the degrees of symmetry of the leading order polynomial of $u-u(x)$.
  In this paper we give a quantitative stratification $\{\cS^k_{\eta,r}\}$ of $u$, which separates
 points based on the number of {\it almost} symmetries of {\it approximate} leading order
 polynomials of $u$ at various scales.  We prove effective estimates on the volume of
 the tubular neighborhood of each $\cS^k_{\eta,r}$, which lead directly to $(n-2+\epsilon)$-Minkowski type estimates for the critical set of $u$.  With some additional regularity
 assumptions on the coefficients of the equation, we refine the estimate to give
new proofs of uniform $(n-2)$-Hausdorff 
measure estimate on the
critical set and singular sets of $u$. 
%

%Mathematical Subject Classification 2010:

%35J15   Second-order elliptic equations

%(maybe) 53B20   Local Riemannian geometry

\end{abstract}

{\small \tableofcontents}
\clearpage

\section{Introduction}

In this paper, we study solutions $u$ to second order linear homogeneous elliptic equations on
subsets of $\dR^n$ and on manifolds with both Lipschitz and smooth coefficients.  
We introduce new quantitative stratification techniques in this context, based on those
 first introduced in \cite{ChNa1,ChNa2}. These techniques allow for new estimates on the critical set
\begin{align}
\Cr(u)\equiv\{x:|\nabla u|=0\}\, 
\end{align}
and more importantly on the {\it effective critical set} 
\begin{align}\label{e:effective_critical}
\cC_r(u)\equiv\cur{x: \inf_{B_r(x)}|\nabla u|^2 < \frac{\epsilon(n)}{r^2}\fint_{\partial B_{2r}(x)} |u-u(x)|^2}\, ,
\end{align}
where $\epsilon(n)$ is a small fixed constant.  That is, if $x\not\in\cC_r(u)$ then not only 
do we have $|\nabla u|(x)\not\eq 0$, but in fact the gradient has some definite size in a ball of definite size around $x$. 

Though most of our results require only a Lipschitz bound on the coefficients, even when applied to harmonic functions on $\dR^n$, the effective estimates are new.  
The Lipschitz bound is sharp in the sense that the results are false under a H\"older assumption.

Because the techniques are local and do not depend on the underlying space on which the equations are 
defined, we will often restrict ourselves to the unit ball $B_1(0)\subseteq \dR^n$.  However, we will point out
  the appropriate modifications needed in the more general situations.  To be specific, we will study equation
s of the form
\begin{gather}\label{eq_Lu}
\L(u)=\partial_i(a^{ij}(x)\partial_j u) + b^i(x) \partial _i u=0
\end{gather}
and
\begin{gather}\label{eq_Lu2}
\L(u)=\partial_i(a^{ij}(x)\partial_j u) + b^i(x) \partial _i u + c(x)u=0\, .
\end{gather}
We will assume  that the coefficients $a^{ij}$ are elliptic and uniformly Lipschitz, and that $b^i$, $c$ are bounded:
\begin{align}\label{e:coefficient_estimates}
(1+\lambda)^{-1}\delta^{ij}\leq a^{ij}\leq (1+\lambda)\delta^{ij}, \, \text{Lip}(a^{ij})\leq \lambda\, , \,
|b^i|,\abs c \leq \lambda\, .
\end{align}

The function $u$ always denotes a weak solution to \eqref{eq_Lu} or \eqref{eq_Lu2}.
Standard elliptic estimates imply that $u\in C^{1,\alpha}$.  Note that if we are interested in
 studying the critical set $\Cr(u)$ then
Lipschitz continuity of the coefficients is essentially the weakest possible regularity assumption
that  we can make.  Indeed, A. Pli\'s (see \cite{plis}) found counterexamples to the unique
 continuation principle for solutions of elliptic equations similar to \eqref{eq_Lu}, where the 
coefficients $a^{ij}$ are H\"older continuous with any exponent strictly smaller than $1$. 
 In such a situation, no reasonable estimates for $\Cr(u)$ 
can hold.

Next, we will give some informal statements of our results; see Sections \ref{ss:quantstrat} 
and \ref{ss:mainresults} for more accurate statements. In the course of doing this, we
will also give a brief review of what was
previously known.  

\paragraph{Harmonic Functions} For simplicity we begin by discussing harmonic functions $\Delta u =0$ on $B_1(0)$.  
The standard fact that such a function is analytic
implies without difficulty that $H^{n-2}(\mathcal C(u)\cap B_{1/2}) <\infty$, if $u$ is not a constant.  

Quantitatively, the standard measurement of {\it nonconstant} behavior of $u$ on a ball $B_r(x)$ is an upper bound on the normalized Almgren frequency defined as:
\begin{align}\label{d:frequency}
%N^u(x,r)\equiv\frac{r \int_{B_r(x)dV} 
%|\nabla u|^2}{\int_{\partial B_r(x)dS} u^2}\, ,\,\, 
\bar N^u(x,r)\equiv\frac{r\int_{B_r(x)} |\nabla u|^2 dV}{\int_{\partial B_r(x)} (u-u(x))^2}\, .
\end{align}
By unique continuation, if $u$ is not constant then both $N^u$ and $\bar N^u$ are 
well defined for positive $r$.  These definitions suggest that harmonic functions might satisfy an estimate of the form
 \begin{gather}\label{eq_estcrit}
H^{n-2}(\Cr(u)\cap B_{1/2})< C(n,\bar N^u(0,1))
 \end{gather}
 In other words, if $u$ is bounded away from being a constant by a definite amount, then 
the critical set can only be so large in the $(n-2)$-Hausdorff sense.  Such an estimate has
 been proved  for the {\it singular set}, i.e. if one restricts to a level set of $u$.  
That is, $H^{n-2}(\Cr(u)\cap B_{1/2}\cap \{u=const\})< C(n,\bar N^u(0,1))$; see \cite{hanhardtlin}.  The paper,   \cite{HLrank}, gives an estimate of this form for the
rank zero sets of harmonic maps. 
 The techniques of  \cite{HLrank}  can  be used to treat case of
sets for equations of the form (\ref{eq_Lu}) (although  this is not pointed out explicitly in
\cite{HLrank}).  In
Theorem \ref{th_Ha} we give a new proof of this bound based on the 
quantitative estimates of Theorem \ref{t:quantstrat}.
(For a slightly earlier proof of the {\it local} finiteness of the $(n-2)$-dimensional
Hausdorff measure of the critical set for equations of the form (\ref{eq_Lu}), see \cite{hoste}.)  More generally the results briefly outlined for harmonic functions hold verbatim for solutions of second order equations with sufficiently smooth coefficients.

In this paper, our main focus is on more effective versions of (\ref{eq_estcrit}).  The estimate
 (\ref{eq_estcrit}) is less than optimal in two primary respects. For general subsets,   a
bound on the 
$(n-2)$-dimensional Hausdorff measure 
 does not prevent the subset from being dense. In fact, even if such a subset is closed, it can
 still be arbitrarily dense. 
Our first statement's tell us that not only is $\Cr(u)$ small, but the tube $B_r(\Cr(u))$ has $(n-2)$-small volume 
for every $r$ in the form of $\Vol(B_r(\Cr(u)))<C_\epsilon r^{2-\epsilon}$ for every $\epsilon$,
 see Theorem \ref{t:crit_lip} for a precise statement.  This is a much stronger statement, which  leads to Minkowski dimension estimates.
 Secondly, as will be seen in Section \ref{ss:mainresults} what we control is 
not just the critical set but the {\it effective} critical set.  That is, we show in Theorem \ref{t:crit_lip} that away from a set of small $(n-2-\epsilon)$-volume (for all $\epsilon$),
every point has a ball of 
definite size in which the gradient has some definite size relative the the nonconstancy of the solution. 
  For details, see subsections \ref{ss:quantstrat}
 and \ref{ss:mainresults}.

%\paragraph{Smooth elliptic equations}
  
%The primary technical construction needed to generalize from the harmonic case to the general elliptic case is the {\it generalized frequency} $\bar F(r)$ of Section \ref{sec_freqell}. 
% This is an almost monotone quanty, in the sense that $e^{Cr}\bar F(r)$ is monotone nondecreasing on some interval $(0,r_0)$; see  Theorem \ref{th_Nellmon}. 
%The function $\bar F(r)$ plays the same role as the frequency for harmonic functions for harmonic functions.  The generalized frequency of Section \ref{sec_freqell} is a variation on a generalized
%frequency constructed in \cite{hanhardtlin}. However, while that quantity is only almost monotone for divergence form operators, the frequency of Section \ref{sec_freqell} is almost monotone for all
%operators of the form \eqref{eq_Lu}.  The techniques work verbatim for solutions of \eqref{eq_Lu2}. However, in this case it is necessary to further restrict the estimate to the zero level set, i.e., we
%obtain $H^{n-2}(\Cr(u)\cap B_{1/2}\cap \{u=0\})< C(n,\bar N^u(0,1))$.

\paragraph{Lipschitz elliptic equations} In reality, the technical heart of this paper concerns
 solutions of elliptic equations with Lipschitz coefficients.  Most of our results, even in the
 smooth coefficient cases, are relatively easy consequences of those in the case where only
 assuming Lipschitz continuity of the coefficients is required.  For example, it is known, see
 \cite{lin}, that $\Cr(u)\cap B_{1/2}$ has Hausdorff dimension $\dim_{Haus}(\Cr(u)\cap B_{1/2})\leq n-2$.
 Although we are not able to improve this to an effective finiteness, we do make advances in two directions.  
First, for all $\epsilon>0$, we do show effective volume estimates of the form
\begin{align}
\label{e:ve}
\Vol(B_r(\Cr(u))\leq \Vol(B_r(\Cr_r(u))\cap B_{1/2})<C(n,\bar N^u(0,1),\epsilon)r^{2-\epsilon}\, .
\end{align}
Among other things this improves $\dim_{Haus}\Cr(u)=n-2$ to $\dim_{Min}\Cr(u)=n-2$.  
That is, the Minkowski dimension of the critical set is at most $n-2$, see Section \ref{ss:quantstrat} 
for precise statements. What is more important, this gives effective estimates for the volume of tubes 
around the critical set, so that even without bounds on $H^{n-2}(\Cr(u))$ in the Lipschitz case, we still have  
very definite effective control over the size of the critical set.  More than that, the corresponding estimate on the effective critical set tells us that away from a $r$-tube of definite volume,
 we have in every ball $B_r(x)$ that $u$ looks close to a linear function after normalization.

The primary technical construction needed to generalize from the harmonic case to the general 
elliptic case is the {\it generalized frequency} $\bar F(r)$ of Section \ref{sec_freqell}.  This is an almost
 monotone quantity, in the sense that $e^{Cr}\bar F(r)$ is monotone nondecreasing on some interval
 $(0,r_0)$; see  Theorem \ref{th_Nellmon}. 
The function $\bar F(r)$ plays the same role as the frequency for harmonic functions.  The generalized frequency of Section \ref{sec_freqell} is a variation on a generalized 
frequency constructed in \cite{galin1,galin2}, which is shown there to be almost monotone
 for operators in divergence form.  Although one can use tricks as in \cite{lin} to apply 
this to nondivergence form operators, instead, by modifying the proof in 
 \cite{galin1,galin2}, we show 
directly in Section \ref{sec_freqell} that the frequency $\bar F$ is almost monotone for 
all operators of the form \eqref{eq_Lu}, which is required for proving (\ref{e:ve}).

\paragraph{Quantitative stratification} More precisely, our primary contribution is the introduction and
 analysis of a quantitative stratification; see Section \ref{ss:quantstrat}. 
 % Based on first tangential behavior of $u$, $\clubsuit$
The standard stratification separates points $x$ in the domain of $u$,  according to the number of independent symmetries
of  the leading order polynomial 
of the Taylor expansion of $u-u(x)$; see \cite{hanlin}.  In particular, this stratification does  not take into account
the degree of the  leading order Taylor polynomial at $x$.
 More precisely, $\cS^k$ consists of those points $x$ such that the leading 
order polynomial $P(y)$ of $u(y)-u(x)$ is a function of at least $n-k$ variables.  For instance, if $u$ has nonvanishing
 gradient at $x$, then the leading order polynomial is linear and therefore $x\in \cS^{n-1}$.

In a manner similar to \cite{ChNa1} and \cite{ChNa2}, we will define a quantitative stratification which
 refines
  the standard stratification. 
  Very roughly, for a fixed $r,\eta>0$ this stratification separates points $x$ based on the number of independent
$\eta$-{\it almost} symmetries 
of an approximate leading order polynomials of $u-u(x)$ at scales  $\geq r$; for a precise definition,
see Section \ref{ss:quantstrat}.

\label{ref_Mink_2} The essential point of this paper is to prove  volume estimates 
on the quantitative stratification, 
as opposed to the
 weaker Hausdorff estimates on the standard stratification.  
As in \cite{ChNa1,ChNa2} these estimates require new techniques which provide 
a quantitative replacement for more traditional blow up arguments.  The new techniques work under Lipschitz constraints
 on the coefficients and, in particular, these arguments give new proofs of the original Hausdorff estimates.  

The key ideas involved in proving the estimates for the quantitative stratification
are {\it quantitative differentiation}, the {\it frequency decomposition} (for the {\it generalized frequency}) which plays the role 
the energy played in \cite{ChNa1,ChNa2}) and {\it cone splitting}.

In general,  precise cone-splitting is the principle that in the presence
of conical structure, nearby symmetries
 interact to create  additional  symmetries.  In the present context, ``$0$-symmetry''
 plays the role of conical structure.
We say that a function $f$  is 
$0$-symmetric at a point, if for some $d>0$, it is homogeneous of degree $d$ at that point. 
 If $f$ is homogeneous of degree $d$ with respect to two distinct points, it follows that $f$ is constant on
lines parallel to the one joining these points  and hence, that $f$ is actually a function of at most $n-1$ variables. 
\footnote{To see this, note that
if $f(x_1,\ldots x_n)$ is homogeneous of degree $d$ with respect to the points $(0,\ldots,0)$
and $(a_1,\ldots,a_n)$, then
$\sum_i x_i\partial_i(f)=\sum_i (x_i-a_i)\partial_i(f)
=d\cdot f$, and so $\sum_i a_i\partial_i(f)=0$. 
}  
 In our terminology, we can rephrase this by saying that if a function is $0$-symmetric at 
two distinct points, then the function is actually $1$-symmetric. We call 
 cone-splitting (as opposed to precise cone-splitting)
 a quantitative version of the above statement.   (In \cite{ChNa2}
 the splitting principle was applied to functions that were simply $0$-homogeneous, that 
is, radially invariant). The frequency decomposition will exploit this by decomposing
 the space $B_1(0)$ based on which scales $u$ looks almost $0$-symmetric.
  On each such piece of the decomposition, and at every scale, nearby points automatically either force
 higher order symmetries or a good covering of the space, and thus the
 estimates of this paper can be proved easily on each piece of the decomposition. 
 The final theorem is obtained by noting that there are far fewer pieces to 
the decomposition than might  {\it apriori} seem possible, a result which follows
 from a {\it quantitative differentiation} argument. 

The Hausdorff estimates on the critical sets of solutions of \eqref{eq_Lu} with
 smooth coefficients will be gotten by combining the estimates on
 the quantitative stratification with an $\epsilon$-regularity type theorem from \cite{hanhardtlin}.

\subsection{The First-Order Stratification}

Even though we will not use the standard stratification in this article, it seems
 appropriate to recall briefly its definition and main properties.
 This should help the reader understand the philosophy underlying the quantitative stratification.

The appropriate notion of stratification in our context is based on 
first order tangent behavior as opposed to the stratifications considered
 in \cite{ChNa1,ChNa2}, which were based on  zeroth order behavior. 
 Specifically, let us first be more careful about the notion of {\it tangent behavior}
 in this context.  We will make all definitions on $\dR^n$, though the analogous 
definitions on manifolds are the same up to the use of an exponential map; 
for example, see \cite{ChNa2}.  We will usually need to work under an assumption
 of nondegeneracy in order to make sense of the tangential behavior:

\begin{definition}
We call a smooth function $u$ nondegenerate if at every $x$ some derivative of some order is nonzero.
\end{definition}
In particular, according to this definition, a constant function is degenerate. 
(This is consistent with the fact that this is a first order stratification). 
On the other hand, any nonconstant analytic function is nondegenerate.  
We now define our tangent maps: 
\begin{definition}\label{d:tangent}
Let $u:B_1(0)\to \dR$ be a smooth nondegenerate function and $r>0$.  
Then we make the following definitions
\begin{enumerate}
\item For $x\in B_{1-r}(0)$ we define
\begin{gather}
 T_{x,r}u(y) = \frac{u(x+ry)-u(x)}{\ton{\fint_{\partial B_1(0)} (u(x+ry)-u(x))^2 }^{1/2}}\, .
\end{gather}
If the denominator vanishes, we set $T_{x,r}=\infty$.
\item For $x\in B_1(0)$ we define
\begin{gather}
 T_{x,0}u(y)=T_x u(y)=\lim_{r\to 0} T_{x,r}u(y)\, .
\end{gather}
\end{enumerate}
\end{definition}

Note that the limits above exist at $x$ as long as $u$ is nondegenerate at $x$. 
 In that case, the limit is unique and, up to rescaling, $T_xu$ is just the leading 
order polynomial of the Taylor expansion of $u-u(x)$ at $x$.  
In particular, $T_xu$ is a homogeneous polynomial, and if $u$ satisfies a second 
order elliptic equation then this polynomial is a homogeneous solution to the 
constant coefficient equation $a^{ij}(x)\partial_i\partial_j T_x=0$.  Hence, up to
 a linear change of coordinates is a homogeneous harmonic polynomial. 
\begin{remark}
For the sake of simplicity, when studying solutions to \eqref{eq_Lu} we will modify
 the definition of $T_{x,r}$ using this linear change of coordinates
 (see Definition \ref{deph_Tell}). In this way, $T_{x,0}u$ will be a 
harmonic homogeneous polynomial. Since the change of variables
 has bi-Lipschitz constant depending only on $\lambda$, from 
the point of view of our results there is no significant difference between these two definitions. 
\end{remark}

Next, we specify what it means for a function to be symmetric,
a key point in the definition of the stratification.

\begin{definition}
Let $u:\dR^n\to \dR$ be a smooth function:
\begin{enumerate}
\item We say $u$ is $0$-symmetric if $u$ is a homogeneous polynomial.
\item We say $u$ is $k$-symmetric if $u$ is $0$-symmetric and there exists a $k$-dimensional
subspace $V$ such that for every $x\in \dR^n$ and $y\in V$ we have that $u(x+y)=u(x)$.
\end{enumerate}
\end{definition}

We can now define the first-order stratification associated to $u$:   

\begin{definition}
Given a smooth nondegenerate function $u:B_1(0)\to \dR$ we define the $k^{th}$-singular stratum of $u$ by
\begin{align}
\cS^k(u)\equiv \{x: T_xu \text{ is not k+1-symmetric}\}\, .
\end{align}
\end{definition}

Let us make a few remarks about some unusual features of this stratification.
They arise from the fact that it  is a {\it first order} stratification.  To begin with, it is usually the case in a stratification that $\cS^{n-1}$ 
has measure zero, that is, that almost every point has $n$-degrees of symmetry.  The issue in general is that for almost every point 
of a nondegenerate function $u$, we have that $T_x u$ is a linear function.  Hence, almost every point has $n-1$ degrees of symmetry,
 and so, $\cS^{n-1}$ has full measure and $\dim \cS^{n-1}=n$.  
%This is a direct consequence of the fact that our stratification is a first order stratification.  
Despite this circumstance, for solutions of \eqref{eq_Lu} and for $k\leq n-2$, we will recover the estimate $\dim \cS^k\leq k$,
 where $\dim$ denotes Hausdorff (or even Minkowski) dimension.

\begin{remark}
 The smoothness assumption on $u$ is a sufficient condition to define the standard stratification, but not a necessary one. 
Indeed, even though solutions to \eqref{eq_Lu} with \eqref{e:coefficient_estimates} are in general only $C^{1,\alpha}$, by 
unique continuation and the maximum principle it is easy to see that for positive $r$, $T_{x,r}u$ is still well-defined and finite.
 
Moreover, by the uniqueness of the tangent maps proved in \cite[theorem 3.1]{han_sing} \footnote{Note that this theorem 
requires as an additional assumption that $u$ does not vanish at infinite order at $x$, which is guaranteed in our context}, 
also $T_{x,0}u$ is well-defined for all $x$.
% This function is a nonzero homogeneous polynomial satisfying the constant coefficient equation
%  \begin{gather}\label{eq_aT}
%   a^{ij}(x) \frac{\partial^2}{\partial y^i \partial y^j} T_x(y) =0\, .
%  \end{gather}

\end{remark}

\subsection{The Quantitative Stratification}\label{ss:quantstrat}

Notice that for solutions to \eqref{eq_Lu} the total singular set $\cS^{n-2}$ is precisely the critical points of $u$, 
namely the points where $|\nabla u|=0$. The goal of this paper is to prove refined estimates on $\cS$ when $u$ 
is not only a nondegenerate function, but also satisfies an elliptic equation.  To do this, an important step is to 
quantify the stratification of the last subsection.  For solutions of elliptic equations, we will prove effective Minkowski type estimates for this 
quantitative stratification.  

To define the quantitative stratification we begin with the following quantitative version of symmetry. 
 Recall the definition of $k$-symmetric and $T_{x,r} u$ from the last subsection.

\begin{definition}
Let $u:B_1(0)\to \dR$ be an $L^2$ function.  We say that $u$ is $(k,\epsilon,r,x)$-symmetric
if there exists a $k$-symmetric polynomial $P$ with $\fint_{\partial B_1(0)} |P|^2=1$ such that
\begin{align}\label{eq_rip}
\fint_{B_1(0)}|T_{x,r}u -P|^2<\epsilon\, .
\end{align}
\end{definition}
\begin{remark}
Note that for harmonic functions and for solutions to \eqref{eq_Lu}, it would make no significant 
difference if we added the assumption that the polynomial $P$ is harmonic. Moreover, we can also replace
 the inequality \eqref{eq_rip} with
\begin{align}\label{eq_rep}
\fint_{\partial B_1(0)}|T_{x,r}u -P|^2<\epsilon'\, .
\end{align}
Indeed, by the doubling conditions in \cite[Corollary 2.2.7]{hanlin}, relation 
\eqref{eq_rep} implies that $u$ is $(k,\epsilon'/n,r,x)$-symmetric. The converse also holds
with the proviso that  in this case, $\epsilon'$ depends on $\epsilon$, $n$ and also on $\bar N^u(0,1)$. 
 Given the definition of frequency function in \eqref{d:frequency}, it is easy to 
see why this second definition is more convenient to use in case $u$ is harmonic, or more 
generally a solution to \eqref{eq_Lu}.
\end{remark}

The above gives a quantitative way of stating that $u$ is {\it almost} $k$-symmetric on $B_r(x)$.  
We are now in a position to define the quantitative stratification:

\begin{definition}
Let $u:B_1(0)\to \dR$ be an $L^2$ function.  Then we define the $(k,\eta,r)$-effective singular stratum by
\begin{align}
\cS^k_{\eta,r}\equiv \{x\in B_1(0): u\text{ is not }(k+1,\eta,s,x)\text{-symmetric} \ \forall s\geq r\}\, .
\end{align}
\end{definition}

The following  properties of the quantitative stratification are immediate.  To begin with,
\begin{align}
\cS^k_{\eta,r}\subseteq \cS^{k'}_{\eta',r'} \text{ if } (k'\leq k, \eta'\leq\eta, r\leq r')\, .
\end{align}
In addition,  we can recover the standard stratification by
\begin{align}
\cS^k = \bigcup_{\eta}\bigcap_{r} \cS^k_{\eta,r}\, .
\end{align}

Our first main result is the following effective Minkowski estimate for $\cS^k_{\eta,r}$,
which holds under the assumption of a frequency bound on $u$, see \eqref{d:frequency}. 
 In particular, we will see that this immediately implies Minkowski dimension control of the 
critical set for solutions of \eqref{eq_Lu}.

\begin{theorem}\label{t:quantstrat}
Let $u:B_1(0)\to\dR$ satisfy \eqref{eq_Lu} and \eqref{e:coefficient_estimates} weakly with
 $\bar N^u(0,1)\leq \Lambda$.  Then
\begin{enumerate}
\item For every $\eta>0$ and $k\leq n-2$ we have 
\begin{align}
{\rm Vol}\ton{B_r(\cS^k_{\eta,r})\cap B_{1/2}(0)}\leq C(n,\lambda,\Lambda,\eta) r^{n-k-\eta}\, .
\end{align}
\item For every $\epsilon>0$ and $0\leq \alpha<1$ there exists $\bar\eta(n,\epsilon,\alpha,\lambda,\Lambda)$ 
such that if $x\not\in \cS^{n-2}_{\eta,r}$ with $\eta<\bar\eta$ then there exists a
 linear function $L(x)$ with $\fint_{\partial B_1(0)}|L|^2=1$ such that $||T_{x,r}u-L||_{C^{1,\alpha}(B_{1/2}(0))}<\epsilon$.
\end{enumerate}
\end{theorem}
\begin{remark}
Note that we have only assumed Lipschitz control on the coefficients $a^{ij}$ and $L^\infty$
 control over the coefficients $b^i$.
\end{remark}

 \begin{remark}
 The theorem continues to hold for solutions of \eqref{eq_Lu2} so long as we only  estimate the volume 
$\Vol\qua{B_r\ton{\cS^k_{\eta,r}\cap u^{-1}(0)}\cap B_{1/2}(0)}$.
 \end{remark}
\begin{remark}
The second item in the theorem implies the following important statement:  there exists
 $\eta(n,\lambda,\Lambda)$ such that $B_r(\Cr(u))\subseteq \cS^{n-2}_{\eta,2r}$. 
 This immediately implies the estimate on  tubular neighborhoods of the critical set, 
which is recorded in Theorem \ref{t:crit_lip} below.
\end{remark}
\begin{remark}
On a Riemannian manifold the constant $C$ should also depend on the 
sectional curvature of $M$ and the volume of $B_1$.  In this case one can
 use local coordinates to immediately deduce the theorem for manifolds from the Euclidean version.  
The estimates \eqref{e:coefficient_estimates} are then with respect to the Riemannian geometry on $M$, 
where $a^{ij}$ and $b^i$ are now tensors on $M$ and $\partial$ is the covariant derivative on $M$. 
\end{remark}

\subsection{The Main Estimates on the Critical Set}\label{ss:mainresults}

Our primary applications of Theorem \ref{t:quantstrat} are to the critical sets of solutions of \eqref{eq_Lu},
 or better to the effective critical sets. Indeed, we will not only give estimates on the set of points
 with vanishing gradient, but also on the set of points where the gradient is small in an appropriate sense.

Given a linear function $L(x)=\ps{\vec L\,}{\, x}$, we say that $L$ is normalized if
\begin{gather}
 \fint_{\partial B_1(0)} \abs{L(x)}^2 =1\, \quad \Longleftrightarrow \, \quad \abs{\vec L} = \beta(n)\, .
\end{gather}

\begin{definition}
 Given $u\in C^1$ and $x$ in its domain, we define
 \begin{gather}
  r_x = \sup\cur{s\geq 0 \ \ \text{s.t. there exists a normalized }L \ \ \text{s.t.} \ \ \norm{T_{x,s} u - L}_{C^1(B_{1/2}(0))}\leq \beta(n)/2} \, .
 \end{gather}
\end{definition}
Given the definition, it is immediate to see that $r_x=0$ if and only if $x$
 is a critical point for $u$. Moreover, we have the estimate
\begin{gather}
 \inf_{y\in B_{1/2}(0)} \cur{\abs{\nabla T_{x,r_x}u(y)}}\geq \beta(n)/2 >0\, .
\end{gather}
We can rephrase the previous estimate in the following form
\begin{gather}
 \inf_{y\in B_{r_x/2}(x)} \cur{\abs{\nabla u(y)}}
\geq \frac{\beta(n)}{2} \frac{\ton{\fint_{\partial B_{r_x}(0)} \qua{u(y)-u(x)}^2 dy}^{1/2} }{r_x}  >0\, .
\end{gather}

Let us give an improved definition of the critical set below. 
 It differs from (\ref{e:effective_critical}) in that for a point $x\not\in \Cr_r(u)$ 
not only is the gradient a definite size, but in fact $u$ looks almost linear after normalization:

\begin{definition}
 Given $r\geq 0$, we define the effective critical set at scale $r$ by
 \begin{gather}
  \Cr_r(u) = \cur{x \ \ \text{s.t.} \ \ r_x \leq r}\, .
 \end{gather}
\end{definition}
It is easy to see that for all $0\leq s \leq r$ we have
\begin{gather}
 \Cr(u) \subseteq \Cr_s(u) \subseteq \Cr_r(u)\, .
\end{gather}

\paragraph{Hausdorff measure and Minkowski content} Before stating the
 results let us quickly recall the notion of
 Hausdorff measure and Minkowski content.  In short, the Hausdorff dimension
 of a set can be small although the set is dense; if the set is
not closed, it can still be arbitrarily dense.  On the other hand, Minkowski type estimates bound not only the set in question, 
but the tubular neighborhood of that set, providing a much more analytically
 effective notion of {\it size}.  Precisely, 
given a set $S\subseteq \dR^n$ its $k$-dimensional Hausdorff measure is defined by
\begin{align}
H^k(S)\equiv \lim_{r\to 0}\sum_{S\subseteq\cup B_{r_i}(x_i):r_i\leq r}w_k r_i^{k}\, .
\end{align}
Hence, the Hausdorff measure is obtained from the most efficient coverings of 
$S$ by balls of arbitrarily small size.  
On the other hand, the Minkowski $k$-content is defined by
\begin{align}
M^k(S)\equiv \lim_{r\to 0}\sum_{S\subseteq\cup B_{r}(x_i)}w_k r^{k}\, .
\end{align}
Hence, the Minkowski $r$-content of $S$ is obtained by covering $S$ with balls of a {\it the same} size, 
$r$, which is then taken to be arbitrarily small.  Equivalently in our situation, it is obtained by 
controlling the volume of tubular neighborhoods of $S$.  The Hausdorff and Minkowski
 dimensions are then defined as the smallest numbers $k$ such that $H^{k'}(S)=0$ or $M^{k'}(S)=0$,
 respectively, for all $k'>k$.  As a simple example note that the Hausdorff dimension of the rationals in $B_1(0)$ is $0$,
 while the Minkowski dimension is $n$.

\paragraph{Main theorem} Let us begin with the following result which is an immediate consequence of 
 Theorem \ref{t:quantstrat} and the remarks following that theorem:

\begin{theorem}\label{t:crit_lip}
Let $u:B_1(0)\to\dR$ satisfy \eqref{eq_Lu} and \eqref{e:coefficient_estimates} weakly 
with $\bar N^u(0,1)\leq \Lambda$.  Then for every $\eta>0$ we have 
\begin{align}
{\rm Vol}(B_r(\Cr_r(u))\cap B_{1/2}(0))\leq C(n,\lambda,\Lambda,\eta)r^{2-\eta}\, .
\end{align}
\end{theorem}
\begin{remark}
This immediately gives us the weaker estimate that Minkowski dimension of $\Cr(u)$ satisfies $\dim_{Min}\Cr(u)\leq n-2$.
\end{remark}
% \begin{remark}
% In fact, according to Theorem \ref{t:quantstrat}, for each $r$ there is a set $\cB_r$ with $\Vol(B_r(\cB_r)\cap B_{1/2}(0))
%\leq C(n,\lambda,\Lambda,\eta)r^{2-\eta}$ such that if $x\not\in \cB_r$ then the gradient of $u$ on $B_r(x)$ has a {\it definite size} relative to $u$. 
 Thus we really have estimates on an effective version of the critical set.
% \end{remark}
\begin{remark}

The theorem still holds for solutions $u$ of \eqref{eq_Lu2}, provided we restrict ourself to the zero level set of $u$.  
That is, in this case we have  $\Vol[B_r(\Cr(u)\cap u^{-1}(0))\cap B_{1/2}(0)]\leq C(n,\lambda,\Lambda,\eta)r^{2-\eta}$.
\end{remark}
\begin{remark}
On a manifold the constant $C$ should also depend on the sectional curvature of $M$ and the volume of $B_1$. 
\end{remark}

\paragraph{$(n-2)$-Hausdorff estimates} As an easy application of Theorem \ref{t:quantstrat} and an important $\epsilon$-regularity 
theorem \cite[Lemma 3.2]{hanhardtlin}, we can show the critical and singular sets have finite $n-2$ measure if we assume the coefficients
 are sufficiently smooth.  Note that this result follows also from the results in \cite{HLrank}. 

\begin{theorem}\label{th_Ha}
Let $u:B_1(0)\to\dR$ satisfy \eqref{eq_Lu} and \eqref{e:coefficient_estimates} weakly with $\bar N^u(0,1)\leq \Lambda$, and such that 
$$||\delta-a||_{C^M},\, ||b||_{C^M}<\lambda\, ,$$ where $M=M(n,\lambda,\Lambda)$.  Then we have that
\begin{align}
H^{n-2}(\Cr(u)\cap B_{1/2}(0))<C(n,\lambda,\Lambda)\, .
\end{align}
\end{theorem}
\begin{remark}
On a manifold the constant $C$ should also depend on the sectional curvature of $M$ and the volume of $B_1$. 
\end{remark}
\begin{remark}
 As mentioned in the introduction, this theorem can also be proved by a simple adaptation of the proof of \cite[theorem A]{HLrank}.
\end{remark}
\begin{remark}
The theorem still holds for solutions $u$ of \eqref{eq_Lu2}, provided that
 we restrict ourselves to the zero level set of $u$.  In this case the result was originally proved 
in \cite[Theorem 3.1]{hanhardtlin} (see also \cite[Theorem 7.2.1]{hanlin}).
\end{remark}

For the sake of clarity,
in giving the proofs, we will at first restrict our study to harmonic functions on $\R^n$.
 Technical details aside, all the ideas needed for the proof of the general case are already
 present in this case.  We will then turn our attention to the general elliptic case, pointing 
out the differences  between the two situations.

\paragraph{Acknowledgement} We are indebted to an anonymous referee for useful comments and for calling to our attention to the reference \cite{HLrank}. 

\section{Harmonic functions}\label{sec_harmonic}

Throughout this section, $u$ will denote a harmonic function on the unit ball, i.e., a function 
$u:B_1(0)\subseteq \R^n \to \R$ which solves
\begin{gather}
 \Delta u =0\, .
\end{gather}

As in \cite{ChNa1,ChNa2} a key tool in the development of a quantitative stratification is the existence of 
an appropriate monotone quantity.  In this context this monotone quantity is the Almgren frequency
 function and its various generalizations, see Section \ref{sec_freqell}.  We begin by introducing the standard 
frequency function. 

\subsection{Almgren's Frequency and Normalized Frequency}
\begin{definition}
 If $u$ is a nonzero harmonic function, for $x\in B_1(0)$ and $r\in (0,1-\abs x)$ we define the Almgren's frequency 
function by:
\begin{gather}
 N^u(x,r)=\frac{r\int_{B_r(x)} \abs{\nabla u}^2 dV}{\int_{\partial B_r(x)} u^2 dS}\, .
 \label{eq_alm}
\end{gather}
If $u$ is nonconstant, we define the normalized version of Almgren's frequency function by:
\begin{gather}
\bar N^u(x,r)=N^{u-u(x)}(x,r)=\frac{r\int_{B_r(x)} \abs{\nabla u}^2 dV}{\int_{\partial B_r(x)} (u-u(x))^2 dS} \label{eq_alm2}\, .
\end{gather}
\end{definition}
\begin{remark}
 As we will see, the frequency function can be used to control the vanishing order of $u$ at each point. 
However, since we are interested in the study of the critical set, not just the singular one, we will need
 information on the vanishing order at $x$ of $u-u(x)$. In this context, the definition of normalized 
frequency in \eqref{eq_alm2} is the natural extension of the standard one.
\end{remark}

An essential property of $N(x,r)$ is that it is invariant under rescaling and blow-ups.  
The normalized frequency $\bar N$, has in addition, the property of remaining unchanged if we add a constant to $u$.  
More generally,
we have the following easily verified lemma.

\begin{lemma}\label{l_blowup}
 Let $\alpha,\beta,\gamma$ be real constants, $\alpha,\beta \neq 0$. If 
$w(x)=\alpha u(\beta x ) + \gamma$, then:
\begin{gather}
 \bar N^u(0,r)= \bar N^w (0,\beta^{-1} r)
\end{gather}
\end{lemma}

The main property of the frequency function is its monotonicity with respect to $r$.
\begin{theorem}
  Let $u$ be a nonconstant harmonic function, and $x\in B_1(0)$.  Then $\bar N(x,r)$
 is monotone nondecreasing with respect to $r$. Moreover, if for some $0\leq r_1<r_2$, 
$\bar N(x,r_1)=\bar N(x,r_2)$, then $u-u(x)$ is a homogeneous harmonic polynomial of degree $d=N(x,r)$ centered at $x$.
\end{theorem}
%\begin{remark}

Here, by definition we say a polynomial $p$ is homogeneous and centered at $x$ if
 $p(y)=\sum_{|\beta|=d} c_\beta (y-x)^{\beta}$, where $\beta$ is a multi-index and $|\beta|\equiv \sum \beta_i$.
%\end{remark}

\begin{proof}
Since $x$ is fixed, it is evident
% that monotonicity
 that the assertions for $\bar N$ are equivalent to those for
$N$.  In that case, they are well-known (see Section \ref{sec_freqell} for a more general computation).
%  of $N$ is equivalent to monotonicity of $\bar N$. 
%   The monotonicity of the frequency function 
%    is well-known (see Section \ref{sec_freqell} for a more general computation).
\end{proof}
\begin{remark}
 Using monotonicity, we can define $\bar N(x,0)=\lim_{r\to 0} \bar N(x,r)$. This quantity has a very
 concrete interpretation. Indeed, it is easy to see that $\bar N(x,0)$ is the degree of the leading 
polynomial $T_{x}u$. By assumption, $u$ is not constant, and thus we deduce the important lower 
bound $\bar N(x,r)\geq \bar N(x,0)\geq 1$ for all $x,r$.
\end{remark}

\begin{remark}\label{rem_dou}
For positive $r$, let $H(x,r)=\fint_{\partial B_r(x)} u^2 dS$. A well-known corollary to the monotonicity 
of $N$ is the following doubling condition on $H$:
\begin{gather}\label{eq_dou}
 H(x,r_2)\leq \ton{\frac{r_2}{r_1}}^{2N(x,r_2)} H(x,r_1)\, .
\end{gather}
By replacing $u$ with $u-u(x)$ we obtain an analogous property for the similarly defined quantity
$\bar H(x,r)=\fint_{\partial B_r(x)} (u-u(x))^2 dS$.  Note that this doubling property has as an immediate 
corollary the unique continuation property for harmonic functions.
\end{remark}

The main results in this paper give estimates that rely on $\bar N^u(0,1)$.  The next lemma proves 
that an upper bound on this quantity implies uniform upper bounds on $\bar N^u (x,r)$, where $x$ and 
$r$ are chosen in such a way that $B_r(x)\Subset B_1(0)$.

\begin{lemma}\label{lemma_strong2bar}
Let $u$ be a nonconstant harmonic function in $B_1(0)\subseteq \R^n$ with $\bar N(0,1)\leq \Lambda$.
 For each positive $\kappa<1$, there exists a function $C(n,\Lambda,\kappa)$ such that for each 
$x\in B_{\kappa}(0)$ and $r\leq \frac 2 3 (1-\kappa)$,
\begin{gather}
 \bar N(x,r) \leq  C(n,\Lambda,\kappa)\, .
\end{gather}
\end{lemma}
\begin{proof}
In \cite[Theorem 2.2.8]{hanlin}, a similar lemma is proved with $N(x,r)$ in place of $\bar N(x,r)$. Here we 
only prove the statement for $\kappa = \frac 1 4$ and $r=\frac 1 2 $, a simple covering and compactness 
argument can be used to prove the general case.

Without loss of generality, we assume $u(0)=0$, and so $N(0,r)=\bar N(0,r)\geq 1$ for all $r\leq 1$. By definition:
\begin{gather}
 \bar N (x,1/2)= \frac{r \int_{B_{1/2}(x)} \abs{\nabla u}^2 dV}{\int_{\partial B_{1/2}(x)} (u-u(x))^2 dS} 
= \frac{r^2 \fint_{B_r(x)} \abs{\nabla u}^2 dV}{n \fint_{\partial B_{1/2}(x)} (u-u(x))^2 dS}
\end{gather}
The mean value theorem for harmonic functions gives:
\begin{gather}
 \fint_{\partial B_r(x)} (u-u(x))^2 dS = \fint_{\partial B_r(x)} u^2 dS - \ u(x) ^2 \geq 0\, .
\end{gather}
Using the doubling conditions in equation \eqref{eq_dou}, we get the estimate
\begin{gather}
 u(x)^2 \leq H(x,1/3)\leq H(x,1/2) (2/3)^{2N(x,1/3)} \, . 
\end{gather}
Thus, we have immediately:
\begin{gather}
  \bar N(x,1/2) = \frac {(1/2)^2} n 
\frac{\fint_{B_{1/2}(x)} \abs{\nabla u}^ 2dV}{\qua{\fint_{\partial B_{1/2}(x)} (u)^2dS}-u(x)^2}
\leq N(x,1/2) \ton{1-(2/3)^{2N(x,1/3)}}^{-1}\, .
\end{gather}
By \cite[Theorem 2.2.8]{hanlin}, we have that $N(x,1/2)\leq C(n,\Lambda)$. 
 In order to conclude the proof, we need to show $N(x,1/3)\geq C(n,\Lambda)$.  
This follows from simple algebraic manipulations. Indeed, by repeated applications
 of standard estimates (or the optimal estimate of \cite[Corollary 2.2.7]{hanlin}), we have
\begin{gather}
 \int_{\partial B_{1/3} (x)} u^2 dS \leq \frac 1 3 (n+2N(x,1/3)) 
\int_{B_{1/3}(x)} u^2 dV \leq C(n,\Lambda) \int_{B_1(0)} u^2dV \leq \frac {C(n,\Lambda)} n 
\int_{\partial B_1(0)} u^2 dS ,
\end{gather}
while by using the doubling conditions in equation \eqref{eq_dou}, we have
\begin{gather}
 \int_{\partial B_1(0)} u^2 dS \leq 12^{n-1+2N(0,1)} \int_{\partial B_{1/12}(0)} u^2 dS\ .
\end{gather}
Finally, by the inclusion $B_{1/12}(0)\subset B_{1/3}(x)$ we have
\begin{gather}
 N(x,1/3)=\frac{(1/3) \int_{B_{1/3}(x)} \abs{\nabla u}^ 2} {\int_{\partial B_{1/3}(x)} u^2}
 \geq C(n,\Lambda) N(0,1/12) \geq C(n,\Lambda)\, .
\end{gather}
\end{proof}

\subsection{Quantitative Rigidity and Cone-Splitting}\label{sec_spl}
In this subsection, we will show that the normalized frequency function can be used to characterize
 the $(k,\epsilon,r,x)$-symmetric points for $u$. Then we will prove the cone-splitting theorem for such points.

As we have seen, a function $u$ is a homogeneous harmonic polynomial of degree $d$ if and only
 if $N(0,r)=d$ for all $r$, or equivalently for $r\in (r_1,r_2)$. Using a simple compactness argument 
and the properties of $\bar N$, we turn this statement into a quantitative characterization of the almost symmetric points.
\begin{theorem}\label{th_N2hom}
 Fix $\eta>0$ and $0\leq \gamma <1$, and let $u$ be a nonconstant harmonic function with 
$\bar N(0,1)\leq \Lambda$. Then there exists a positive $\epsilon=\epsilon(n,\Lambda,\eta,\gamma)$ such that if
\begin{gather}
\bar N(0,1)-\bar N(0,\gamma)< \epsilon\, ,
\end{gather}
then $u$ is $(0,\eta,1,0)$-symmetric.
\end{theorem}
\begin{proof}
 Suppose by contradiction that there exists a sequence of harmonic 
functions $u_i$ with $\bar N^{u_i}(0,1)\leq \Lambda$, 
$\bar N^{u_i}(0,1)-\bar N^{u_i}(0,\gamma)< \frac 1 i$ but all the $u_i$ are not $(0,\eta,1,0)$-symmetric.

 From the invariance under rescaling of the frequency and of 
the concept of almost symmetry, we can assume without loss of 
generality that $\fint_{\partial B_1(0)} u_i^2 dS =1$ and $u_i(0)=0$ 
for all $i$, i.e. $u_i = T_{0,1} u_i$. Thus by compactness, $u_i$ converges 
weakly in $W^{1,2}(B_1(0))$ to a harmonic function $u$, and by elliptic estimates, 
the convergence is also in the local $C^1(B_1)$ sense. Using the theory of traces 
for Sobolev spaces, it is easily seen that $\fint_{\partial B_1(0)} u^2 dS =1$ and that 
$N^u (0,1)\leq \Lambda$. Moreover, using the monotonicity of $\bar N$ and passing to 
the limit in $n$ we have:
\begin{gather}
 \bar N^u(0,1)-\bar N^u(0,\gamma)=0\, .
\end{gather}
This implies that $u$ is a harmonic homogeneous polynomial, and since
\begin{gather}
 \lim_{i\to \infty}\fint_{\partial B_1(0)} (u_i-u)^2 dS =0\, ,
\end{gather}
we obtain a contradiction.
\end{proof}
\begin{remark}
 By the invariance properties of $\bar N$, it is evident that we can replace
 the hypothesis $\bar N(0,1)-\bar N(0,\gamma)<\epsilon$ with $\bar N(0,r)-\bar N(0,\gamma r)<\epsilon$
 and obtain that $u$ is $(0,\eta,r,0)$-symmetric.
\end{remark}
\begin{remark}[Quantitative Differentiation]\label{qd}
Note that the above lemma automatically provides a control on the number 
of scales at which $u$ is not $(0,\eta,r,x)$-symmetric. Indeed, set 
$r_i=\gamma^i$ for some $0<\gamma<1$. By monotonicity, there can be only a definite 
 number of $i$'s such that $\bar N(x,\gamma^i)-\bar N(x,\gamma^{i+1}) \geq \epsilon$. 
Then $u$ is $(0,\eta,\gamma^i,x)$-symmetric, for all the ``good'' values of $i$.
\end{remark}

In order to describe how two almost symmetric points interact, we briefly recall 
what happens to homogeneous polynomials.
\begin{proposition}\label{prop_poly}
 Let $P:\R^n\to \R$ be a harmonic polynomial of degree $d$, homogeneous 
with respect to the origin. Suppose also that $P$ is symmetric with respect to the $k$ dimensional subspace $V$. Then
\begin{enumerate}
 \item $P$ is of degree $1$ if and only if it is $n-1$ symmetric
 \item if $P$ is not $n-1$ symmetric, and $P$ is also $0$-symmetric 
with respect to $x\not \in V$, then $P$ is $k+1$-symmetric with respect to $\operatorname{span}(V,x)$.
\end{enumerate}
\end{proposition}
\begin{proof}
Since $P$ is supposed to be harmonic, (1) is straightforward to prove. (2) is a standard exercise in 
algebra. (A similar computation is carried out in the proof of \cite[theorem 4.1.3]{hanlin}).
\end{proof}

By using a compactness argument similar to the one used for Theorem \ref{th_N2hom}, we
can  turn the previous proposition into a quantitative  cone-splitting theorem for almost 
symmetric harmonic functions. As always, note that this statement is scale invariant.
\begin{theorem}\label{th_qspl}
Fix some positive 
$\epsilon,\tau$ and $0<r\leq 1$ and let $k\leq n-2$. 
Let $u$ be a harmonic function with $\bar N(0,1)\leq \Lambda$.
There exists a positive
 $\delta=\delta (n,\Lambda,\tau,\epsilon,r)$ such that if
\begin{enumerate}
 \item $u$ is $(k,\delta,r,0)$-symmetric with respect to the $k$-dimensional subspace $V$,
 \item for some $x\in B_{r}(0)\setminus B_\tau (V)$, $u$ is $(0,\delta,r,x)$-symmetric,
\end{enumerate}
then $u$ is also $(k+1,\epsilon,1,0)$-symmetric.
\end{theorem}
\begin{proof}
We set up the usual contradiction argument. In particular, choose a sequence $u_i$ 
of harmonic functions with $u_i(0)=0$ and $\fint_{\partial B_1(0)} u_i^2 dS =1$ 
which is $(k,i^{-1},r,0)$-symmetric with respect to $V_i$ and $(0,i^{-1},r,0)$-symmetric 
with respect to $x_i$. The bound on the frequency implies that $u_i$ is bounded in
 $W^{1,2}(B_1(0))$. Thus, after passing to a subsequence if necessary, 
we can assume that $u_i \to u$, $V_i\to V$ and $x_i\to x \not \in V$.

On the other hand, by hypothesis $T_{0,r} u_i$ converges to a $k$-symmetric 
normalized homogeneous polynomial $P$. By the doubling conditions in equation \eqref{eq_dou}, we have
\begin{gather}
 \fint_{\partial B_r} u_i^2 dS \geq r^{2\Lambda}>0 \ ,
\end{gather}
so $P=u$. In a similar fashion, $u$ is also a $(0,x)$-symmetric polynomial,
 and by Proposition \ref{prop_poly} $P$ is $(k+1,0)$-symmetric.

Since $u_i$ converges to $P$ in $W^{1,2}(B_1(0))$, we obtain a contradiction.
\end{proof}

The following equivalent version of Theorem \ref{th_qspl}
 will be useful in subsequent sections.

\begin{corollary}\label{cor_symsum}
Fix some positive 
$\eta,\tau$ and $0<r\leq 1$ and let $k\leq n-2$.
Let $u$ be a harmonic function with $\bar N(0,1)\leq \Lambda$. 
There exists $\epsilon=\epsilon(n,\Lambda,\tau,\eta,r)>0$ such that if
\begin{enumerate}
 \item $u$ is $(0,\epsilon,r,0)$-symmetric,
 \item for every subspace $V$ of dimension $\leq k$, there exists 
$x\in B_{r}(0)\setminus B_{\tau}(V)$ such that $u$ is \hbox{$(x,\epsilon,r,0)$-symmetric},
\end{enumerate}
then $u$ is also $(k+1,\eta,1,0)$-symmetric.
\end{corollary}
The proof of this corollary is via a simple induction argument which will be omitted.
For similar arguments see \cite{ChNa1,ChNa2}

We close this subsection with the proof of point (2) in Theorem \ref{t:quantstrat}. 
This proposition is essential for  turning estimates on the singular strata $\cS^{k}_{\eta,r}$
 into estimates on the critical set. In fact,
we show the following.
\begin{proposition}\label{prop_qcrit}
 Let $u$ be harmonic with $\bar N(x,r)\leq \Lambda$.
Fix  $\epsilon>0$ and $k\in \N$. There exists $\bar \eta=\bar \eta(n,k,\epsilon,\Lambda)>0$ 
such that if $u$ is $(n-1,\bar \eta,r,x)$-symmetric, then
 \begin{gather}
  \norm {T_{x,r}u - L}_{C^{k} (B_{1/2}(0))}\leq \epsilon\, ,
 \end{gather}
where $L$ is a linear polynomial with $\fint_{\partial B_r} \abs L ^2 dS =1$.
 In particular, by choosing $k=1$ and $\epsilon$ small enough, 
there exists  $\eta=\eta(n,\Lambda)$ such that if $u$ is $(n-1,\eta,r,x)$-symmetric then $r_x\geq r$.
\end{proposition}
\begin{proof}
The proof is a simple application of the usual contradiction-compactness argument.
 Note that, by elliptic estimates, if $u_i$ converges to $u$ in the weak $W^{1,2}(B_1(0))$ sense,
 then for all $K\Subset B_1(0)$ the convergence is also in the metric of $C^{\infty}(K)$.
 Note also that if $L$ is a linear function with $\fint_{\partial B_1(0)} \abs L^2 dS =1$, 
then $\nabla L$ is a vector of fixed positive length. Thus the second part of the statement
 can be proved by choosing $\epsilon = \abs{\nabla L}/2$.
\end{proof}

\subsection{The Frequency Decomposition}
We are now ready to prove Theorem \ref{t:quantstrat}.  The proof employs the same
techniques  that were introduced for corresponding purposes in \cite{ChNa1,ChNa2};
the reader may wish to consult these references. Instead of proving the statement for any $r>0$, 
we fix a $0<\gamma<1$ and restrict ourselves to the case $r=\gamma^j$ for any $j\in \N$.
 It is evident that the general statement follows. For the reader's convenience we restate 
Theorem \ref{t:quantstrat} under this convention.

\begin{theorem}\label{th_main_proof}
Let $u:B_1(0)\to\dR$ be a harmonic function with $\bar N^u(0,1)\leq \Lambda$. 
Then for every $j\in \N$, $\eta>0$ and $k\leq n-2$, there exists $0<\gamma(n,\eta,\Lambda)<1$ such that
\begin{align}
{\rm Vol}\ton{B_{\gamma^j}(\cS^k_{\eta,\gamma^j})\cap B_{1/2}(0)}\leq C(n,\Lambda,\eta) \ton{\gamma^j}^{n-k-\eta}\, .
\end{align}
\end{theorem}

The scheme of the proof is the following: for some convenient $0<\gamma<1$ we prove
 that there exists a covering of $S^{k}_{\eta,\gamma^j}$ made of nonempty open sets in the collection 
$\{\cC^k_{\eta,\gamma^j}\}$. Each set $\cC^k_{\eta,\gamma^j}$ is the union of a controlled number of balls
 of radius $\gamma^j$. Using Remark \ref{qd} (Quantitative differentiation)
it will follow that the number of nonempty elements in each family has a
%$\clubsuit$ 
 bound 
of the form $j^D$, 
%$\clubsuit$
for some constant $D(n,\eta,\Lambda)>1$.
This will give the desired volume bound. In particular:

\begin{lemma}[Decomposition Lemma]\label{lemma_dec}
 There exists $c_0(n),c_1(n)>0$ and $D(n,\eta,\Lambda)>1$ such that for every $j\in \N$,
\begin{enumerate}
 \item $\cS^k_{\eta,\gamma^j}\cap B_{1/2}(0)$ is contained in the union of at most 
$j^D$ \textit{nonempty} open sets $C^k_{\eta,\gamma^j}$.
 \item Each $C^k_{\eta,\gamma^j}$ is the union of at most 
$(c_1\gamma ^{-n})^D (c_0\gamma^{-k})^{j-D}$ balls of radius $\gamma^j$.
\end{enumerate}
\end{lemma}
Once this Lemma is proved, Theorem \ref{th_main_proof} easily follows.
\begin{proof}[Proof of Theorem \ref{th_main_proof}]
 Let $\gamma=c_0^{-2/\eta}<1$. Since we have a covering of
 $\cS^k_{\eta,\gamma^j}\cap B_{1/2}(0)$ by balls of radius $\gamma^j$, 
it is easy to get a covering of $B_{\gamma^j}\ton{\cS^k_{\eta,\gamma^j}}\cap B_{1}(0)$. 
In fact it is sufficient to double the radius of the original balls. Now it is evident that
\begin{gather}
 \Vol\qua{B_{\gamma^j}\ton{\cS^k_{\eta,\gamma^j}}\cap B_{1/2}(0)} 
\leq j^D \ton{(c_1\gamma^{-n})^D (c_0\gamma^{-k})^{j-D}} \omega_n 2^n \ton{\gamma^j}^n\, ,
\end{gather}
where $\omega_n$ is the volume of the $n$-dimensional unit ball. By plugging in the simple rough estimates
\begin{gather}
 j^D \leq c(n,\Lambda,\eta)\ton{\gamma^j}^{-\eta/2}\, ,\\
\notag (c_1\gamma^{-n})^D(c_0\gamma^{-k})^{-D}\leq c(n,\Lambda,\eta)\, ,
\end{gather}
and using the definition of $\gamma$, we obtain the desired result.
\end{proof}

\paragraph{Proof of the Decomposition Lemma}
Now we turn to the proof of the Decomposition Lemma. In order to do this, we define a new quantity 
which measures the non-symmetry of $u$ at a certain scale.
\begin{definition}
 Given $u$ as in Theorem \ref{th_main_proof}, $x\in B_{1}(0)$ and $0<r<1$, define
\begin{gather}
 \cN(u,x,r) =\inf\{\alpha\geq 0 \ \ s.t. \ \ u \text{ is } (0,\alpha,r,x)\text{-symmetric}\}\, .
\end{gather}
\end{definition}
Given $\epsilon>0$, we divide the set $B_{1/2}(0)$ into two subsets according to the behaviour 
of the points with respect to their quantitative symmetry. 
\begin{align}
 H_{r,\epsilon}(u)=\{x\in B_{1/2}(0) \ s.t. \ \cN(u,x,r)\geq \epsilon\}\, , \\
\notag L_{r,\epsilon}(u)=\{x\in B_{1/2}(0) \ s.t. \ \cN(u,x,r)< \epsilon\}\, .\
\end{align}
Next, to each point $x\in B_{1/2}(0)$ we associate a $j$-tuple $T^j(x)$ of numbers $\{0,1\}$
 in such a way that the $i$-th entry of $T^j$ is $1$ if $x\in H_{\gamma^i, \epsilon}(u)$, and zero otherwise.
 Then, for each fixed $j$-tuple $\bar T^j$, set:
\begin{gather}
 E(\bar T^j) = \{x\in B_{1/2}(0) \ \ s.t. \ \ T^j(x)=\bar T^j\}\, .
\end{gather}
Also, we denote by $T^{ j-1}$, the $(j - 1)$-tuple obtained from $T^ j$ by dropping the last entry,
 and define $\abs{T^j}$ to be the number of entries that are equal to $1$ the $j$-tuple $T^j$.

We will build the families $\{C^k_{\eta,\gamma^j}\}$ by induction on $j$ in the following way.
For $a=0$, $\{C^k_{\eta,\gamma^0}\}$ consists of the single ball $B_{1}(0)$.

\paragraph{Induction step}
For fixed $a\leq j$, consider all the $2^a$ $a$-tuples $\bar T^a$. Label the sets in the family
 $\{C^k_{\eta,\gamma^a}\}$ by all the possible $\bar T^a$. We will build $C^k_{\eta,\gamma^a}(\bar T^a)$ 
inductively as follows.  For each ball $B_{\gamma^{a-1}}(y)$ in $\{C^k_{\eta,\gamma^{a-1}}(\bar T^{a-1})\}$ 
take a minimal covering of $B_{\gamma^{a-1}}(y)\cap \cS^{k}_{\eta,\gamma^j} \cap E(\bar T^a)$ by balls 
of radius $\gamma^a$ centered at points in $B_{\gamma^{a-1}}(x)\cap \cS^k_{\eta,\gamma^j}\cap E(\bar T^a)$.
 Note that it is possible that for some $a$-tuple $\bar T^a$, the set $E(\bar T^a)$ is empty, and in this
 case $\{C^k_{\eta,\gamma^{a}}(\bar T^{a})\}$ is the empty set.

Now we need to prove that the minimal covering satisfies points 1 and 2 in Lemma \ref{lemma_dec}.
\begin{remark}
The value of $\epsilon>0$ will be chosen according to Lemma \ref{lemma_cov}. For the moment,
we take it to be an arbitrary fixed small quantity.
\end{remark}

\paragraph{Point 1 in Lemma}
As we will see below, we can use the monotonicity of $\bar N$ to prove that for every 
$\bar T^j$, $E(\bar T^j)$ is empty if $\abs{\bar T^j}\geq D$. Since for every $j$ there 
are at most $\binom j D\leq j^D$ choices of $j$-tuples with $\abs{\bar T^j}\leq D$,
 the first point will be proved.

\begin{lemma}\label{lemma_K}
 There exists  $D=D(\epsilon,\gamma,\Lambda,n)$ such that $E(\bar T^j)$ is empty if 
$\abs{\bar T^j}\geq D$. 
\end{lemma}

In what  follows, we will fix $\epsilon$ as a function of $\eta,\Lambda,n$. Thus, $D$ will 
actually depend only on these three variables.
\begin{proof}
 Recall that $\bar N(x,r)$ is monotone  nondecreasing with respect to $r$, and, by Lemma 
\ref{lemma_strong2bar}, $\bar N(x,1/3)$ is bounded above by a function $C(n,\Lambda)$. For $s<r$,
we set
\begin{gather}
 \cW_{s,r}(x)=\bar N(x,r)-\bar N(x,s)\geq 0\, .
\end{gather}
If $(s_i,r_i)$ are \textit{disjoint} intervals with $\max\{r_i\}\leq 1/3$, then by monotonicity of $\bar N$:
\begin{gather}\label{eq_sum}
 \sum_i \cW_{s_i,r_i}(x)\leq \bar N (x,1/3) -\bar N (x,0) \leq C(n,\Lambda) -1\, .
\end{gather}

Let $\bar i$ be such that $\gamma^{\bar i} \leq 1/3$, and consider intervals of the form 
$(\gamma^{i+1},  \gamma^{i})$ for $i=\bar i,\bar i+1,...\infty$. By Theorem \ref{th_N2hom} 
and Lemma \ref{lemma_strong2bar}, there exists a $0<\delta=\delta(\epsilon,\gamma,\Lambda,n)$
 independent of $x$ such that
\begin{gather}
 \cW_{\gamma^{i+1},\gamma^{i}}(x)\leq \delta \implies u \text{ is }(0,\epsilon,\gamma^{i},x)\text{-symmetric}\, .
\end{gather}
In particular $x\in L_{ \gamma^{i},\epsilon}$, so that, if $i\leq j$, the $i$-th entry of $T^j$ is necessarily zero. 
By equation \eqref{eq_sum}, there can be only a finite number of $i$'s such that 
$\cW_{\gamma^{i+1}, \gamma^i}(x)>\delta$, and this number $D$ is bounded by:
\begin{gather}\label{eq_estK}
 D\leq \frac{C(n,\Lambda)-1}{\delta(\epsilon,\gamma,\Lambda,n)}\, .
\end{gather}
This completes the proof.
\end{proof}

\paragraph{Point 2 in Lemma}
The proof of the second point in Lemma \ref{lemma_dec} is mainly based on
 Corollary \ref{cor_symsum}. In particular, for fixed $k$ and $\eta$ in the 
definition of $\cS^k_{\eta,\gamma^j}$, choose $\epsilon$ in such a way
 that Corollary \ref{cor_symsum} can be applied with $r=\gamma^{-1}$ and $\tau = 7 ^{-1}$. 
Then we can restate the lemma as follows:
\begin{lemma}\label{lemma_cov}
 Let $\bar T^j_a =0$. Then the set $A=S^{k}_{\eta,\gamma^j}\cap B_{\gamma^{a-1}}(x)\cap E(\bar T^j)$ 
can be covered by $c_0(n)\gamma^{-k}$ balls centered in $A$ of radius $\gamma^{a}$.
\end{lemma}
\begin{proof}
 First of all, note that since $\bar T^j_a =0$, all the points in $E(\bar T^j)$ are in $L_{\epsilon, \gamma^a}(u)$.

The set $A$ is contained in $B_{7^{-1}\gamma^a}(V^k)\cap B_{\gamma^{a-1}}(x)$ for some 
$k$-dimensional subspace $V^k$. Indeed, if there were a point $x\in A$, such that  
$x\not\in B_{7^{-1} \gamma^a}(V^k)\cap B_{\gamma^{a-1}}(x)$, then by Corollary
 \ref{cor_symsum} and Lemma \ref{lemma_strong2bar}, $u$ would be $(k+1,\eta,\gamma^{a-1},x)$-symmetric. 
This contradicts $x\in \cS^k_{\eta,\gamma^j}$. By standard geometry, it follows that $V^k \cap B_{\gamma^{a-1}}(x)$ 
can be covered by $c_0(n)\gamma^{-k}$ balls of radius $\frac 6 7 \gamma^a$, and by the triangle inequality
 it is evident that the same balls with radius $\gamma^a$ cover the whole set $A$.
\end{proof}

If instead $\bar T^j_a =1$, then without any effort we can say that 
$A=S^{k}_{\eta,\gamma^j}\cap B_{a-1}(x)\cap E(\bar T^j)$ can be covered by 
$c_0(n)\gamma^{-n}$ balls of radius $\gamma^a$. Now by a simple induction argument the proof is complete.
\begin{lemma}
 Each (nonempty) $C^k_{\eta,\gamma^j}$ is the union of at most 
$(c_1\gamma ^{-n})^D \cdot (c_0\gamma^{-k})^{j-D}$ balls of radius $\gamma^j$.
\end{lemma}
\begin{proof}
 Fix a sequence $\bar T^j$ and consider the set $C^k_{\eta,\gamma^j}(\bar T^j)$. 
By Lemma \ref{lemma_K}, we can assume that $\abs {\bar T^j}\leq D$, otherwise
 there is nothing to prove since $C^k_{\eta,\gamma^j}(\bar T^j)$ would be empty.

Consider that for each step $a$, in order to get a (minimal) covering of 
$B_{\gamma^{a-1}}(x)\cap S^{k}_{\eta,\gamma^i}\cap E(\bar T^j) $ for 
$B_{\gamma^{a-1}}(x)\in C^k_{\eta,\gamma^{a-1}}(\bar T^j)$, we require at most
 $(c_0 \gamma^{-k})$ balls of radius $\gamma^{a}$ if $\bar T^j_a=0$ or $(c_0\gamma^{n})$ 
otherwise. Since the latter situation can occur at most $D$ times, the proof is complete.
\end{proof}

\subsection{Minkowski Type Estimates on the Critical Set}

Apart from the volume estimate, Theorem \ref{t:quantstrat} has a useful corollary for 
measuring the size of the critical set.  Indeed, by Proposition \ref{prop_qcrit},
 the critical set of $u$ is contained in $\cS^{n-2}_{\epsilon,r}$, thus we have proved 
Theorem \ref{t:crit_lip} for harmonic functions:

\begin{corollary}\label{cor_main}
 Let $u:B_1(0)\to \R$ be a harmonic function with $\bar N^u(0,1)\leq \Lambda$. 
Then, for every $\eta>0$,
\begin{align}\label{eq_a}
{\rm Vol}(B_r(\Cr_r(u))\cap B_{1/2}(0))\leq C(n,\Lambda,\eta)r^{2-\eta}\, .
\end{align}
\end{corollary}
\begin{proof}
By Proposition \ref{prop_qcrit}, for $\eta>0$ small enough, we have the inclusion
\begin{gather}
 B_{r/2}(\Cr_r(u)) \subseteq \cS^{n-2}_{\eta,r}\, .
\end{gather}
Using Theorem \ref{t:quantstrat}, we obtain the desired volume estimate 
for $\eta$ sufficiently small. However, since
\begin{gather}
 \Vol(B_r(\Cr_r(u))\cap B_{1/2}(0))\leq \Vol(B_{1/2}(0))\, ,
\end{gather}
it is evident that if \eqref{eq_a} holds for some $\eta$, then a similar statement holds also for any $\eta'\geq \eta$.
\end{proof}

\begin{remark}
 As already mentioned in the introduction, this volume estimate on the critical set and
 its tubular neighborhoods immediately implies that $\dim_{Mink}(\Cr(u))\leq n-2$. This result is clearly optimal.
\end{remark}

\subsection[The Uniform (n-2)-Hausdorff Bound for the Critical Set]{The Uniform 
$(n-2)$-Hausdorff
 Bound for the Critical Set}

By combining the results of the previous sections with an $\epsilon$-regularity theorem from \cite{hanhardtlin},
 in this subsection we give a new proof of an effective uniform bound on the $(n-2)$-dimensional
Hausdorff measure of $\Cr(u)$. The bound will not depend on $u$ itself, but only on the normalized 
frequency $\bar N^u (0,1)$.  Specifically, the proof will be obtained by combining the 
$(n-3+\eta)$-Minkowski type estimates available for $\cS^{n-3}_{\eta,r}$ with the following $\epsilon$-regularity lemma.  
The lemma states that if a harmonic function $u$ is sufficiently close to a homogeneous harmonic
 polynomial of only $2$ variables, then the whole critical set of $u$ has a definite upper bound
on its $(n-2)$-dimensional Hausdorff measure.

As noted in the introduction, these results also follow from an 
adaptation of the techniques used in \cite{HLrank}

\begin{lemma}\label{lemma_n-2Ha}\cite[Lemma 3.2]{hanhardtlin}
Let $P$ be a homogeneous harmonic polynomial with exactly $n-2$ symmetries in $\R^n$. 
Then there exist positive constants $\epsilon$ and $\bar r$ depending on $P$, such that for any 
$u \in C^{2d^2}(B_1(0))$, if
\begin{gather}
 \norm{u-P}_{C^{2d^2}(B_1)}<\epsilon\, ,
\end{gather}
then for all $r\leq \bar r$:
\begin{gather}
 H^{n-2}(\nabla u^{-1}(0)\cap B_r(0))\leq c(n)(d-1)^2 r^{n-2}\, .
\end{gather}
\end{lemma}

It is not difficult to see that, if we assume $u$ harmonic in $B_1$ with $\bar N^u(0,1)\leq \Lambda$, 
then $\epsilon$ and $\bar r$ can be chosen to be independent of $P$, but dependent only on $\Lambda$.
 Indeed, up to rotations and rescaling, all polynomials with $n-2$ symmetries in $\R^n$ of degree $d$
 look like $P(r,\theta,z)=r^d \cos(d\theta)$, where we used cylindrical coordinates for $\R^n$. 
Combining this with elliptic estimates yields the following corollary.

\begin{corollary}\label{cor_ereg}
Let $u:B_1\to \R$ be a harmonic function with $\bar N(0,1)\leq \Lambda$. Then there exist positive 
constants $\epsilon(\Lambda,n)$ and $\bar r(\Lambda,n)$ such that if there exists a normalized
 homogeneous harmonic polynomial $P$ with $n-2$ symmetries such that
\begin{gather}
 \norm{T^u_{0,1}-P}_{L^2(\partial B_1)}<\epsilon\, ,\, \fint_{\partial B_1(0)} P^2 = 1\, ,
\end{gather}
then for all $r\leq \bar r$:
\begin{gather}
 H^{n-2}(\nabla u^{-1}(0)\cap B_r(0))\leq c(\Lambda,n)r^{n-2}\, .
\end{gather}
\end{corollary}

To prove the effective bound on the $(n-2)$-dimensional  Hausdorff measure, we combine the
 Minkowski type estimates of Theorem \ref{t:quantstrat} with the above corollary.  Using the quantitative 
stratification, we will use an inductive construction to split the critical set at different scales into a 
good part, the points where the function is close to an $(n-2)$-symmetric polynomial, and a bad part, 
whose tubular neighborhoods have definite bounds. Since we have estimates on the whole critical 
set in the good part, we do not have to worry any longer when we pass to a smaller scale. 
As for the bad part, by induction, we start  the process over and split it again into a good and a bad part. 
By summing the various contributions to the 
$(n-2)$-dimensional Hausdorff measure given by the good parts, we prove the following theorem:
\begin{theorem}\label{th_n-2h}
 Let $u$ be a harmonic function in $B_1(0)$ with $\bar N(0,1)\leq \Lambda$. There exists a constant 
$C(\Lambda,n)$ such that
\begin{gather}
 H^{n-2}(\Cr(u)\cap B_{1/2}(0))\leq C(n,\Lambda)\, .
\end{gather}
\end{theorem}
\begin{proof}
Note that by Lemma \ref{lemma_strong2bar}, for every $r\leq 1/3$ and $x\in B_{1/2}(0)$, the functions 
$T_{x,r}u$ have frequency uniformly bounded by $N^{T_{x,r}u}(0,1)\leq C(\Lambda,n)$. 
This will allow us to apply Corollary \ref{cor_ereg} to each $T_{x,r}u$ and obtain uniform 
constants $\epsilon(\Lambda,n)$ and $\bar r(\Lambda,n)$ such that the conclusion of the 
Corollary holds for all $x\in B_{1/2}(0)$ and $r\leq \bar r$.

Now fix $\eta>0$ to be the minimum of $\eta(n,\Lambda)$ from Proposition \ref{prop_qcrit} and 
$\epsilon(n,\Lambda)$ from Corollary \ref{cor_ereg}. Let $0<\gamma\leq 1/3$ and define the following sets:
\begin{gather}
 \Cr^{(0)}(u)=\Cr(u)\cap \ton{S^{n-2}_{\eta,1}\setminus S^{n-3}_{\eta,1}}\cap B_{1/2}(0)\, .\\
 \Cr^{(j)}(u)=\Cr(u)\cap \ton{S^{n-2}_{\eta,\gamma^j}
\setminus S^{n-3}_{\eta,\gamma^j}} \cap S^{n-3}_{\eta,\gamma^{j-1}}\cap B_{1/2}(0)\, .
\end{gather}
We decompose the critical set as follows:
\begin{gather}
 \Cr(u)\cap B_{1/2}(0)=\bigcup_{j=0}^\infty \Cr^{(j)}(u) \bigcup \ton{\Cr(u) \bigcap_{j=1}^\infty\, . 
S^{n-3}_{\eta,\gamma^j}}\, .
\end{gather}
It is evident from Theorem \ref{th_main_proof} that
\begin{gather}
 H^{n-2}\ton{\Cr(u) \bigcap_{j=1}^\infty S^{n-3}_{\eta,\gamma^j}\cap B_{1/2}(0)}=0\, .
\end{gather}
As for the other set, we will prove that
\begin{gather}
 H^{n-2} \ton{\bigcup_{j=0}^k \Cr^{(j)}(u)}\leq C(\Lambda,n,\eta)\sum_{j=0}^{k} \gamma^{(1-\eta)j}\, .
\end{gather}
Using Corollary \ref{cor_ereg} and a simple covering argument, it is easy to see that this 
statement is valid for $k=0$.

Choose a covering of the set $\Cr^{(k)}(u)$ by balls centered at $x_i \in \Cr^{(k)}(u)$ of 
radius $\gamma^k \bar r$, such that the same balls with half the radius are disjoint. 
Let $m(k)$ be the number of such balls. By the volume estimates in Theorem \ref{t:quantstrat}, we have
\begin{gather}
 m(k)\leq C(\eta,\Lambda,n)\gamma^{(3-\eta-n)k}\, .
\end{gather}
By construction of the set $\Cr^{(k)}(u)$, for each $x_i$ there exists a scale 
$s\in [\gamma^k,\gamma^{k-1}]$ such that for some normalized homogeneous 
polynomial of two variables $P$, we have
\begin{gather}
 \norm{T_{x_i,s}u-P}_{L^2(\partial B_1)}<\eta\, .
\end{gather}
Note that since $u$ is harmonic, we can assume without loss of generality that $P$ is 
harmonic as well. Indeed, if $\eta$ is small enough, we can find a homogeneous 
harmonic polynomial $P'$ such that $\norm{P-P'}_{L^2(\partial B_1)}<\eta$.

Using Corollary \ref{cor_ereg} we can deduce that
\begin{gather}
 H^{n-2}\ton{\nabla u^{-1}(0)\cap B_{\gamma^k \bar r}(x_i)}\leq C(\Lambda,n)\gamma^{(n-2)k}\, .
\end{gather}
Therefore,
\begin{gather}
 H^{n-2}\ton{\Cr^{(k)}(u)}\leq C(\Lambda,n,\eta)\gamma^{(1-\eta)k}\, .
\end{gather}
Since $0<\gamma,\eta<1$, the proof is complete.
\end{proof}

\section{Elliptic equations}\label{sec_ell}
With appropriate  modifications, the results proved for harmonic functions are valid  for 
solutions to elliptic equations of the form \eqref{eq_Lu} with conditions \eqref{e:coefficient_estimates}.
 Indeed, a Minkowski type estimate of the form given in Theorem \ref{th_main_proof} and Corollary \ref{cor_main} 
(in which there is an arbitrarily small positive loss in the exponent) remains valid without any 
further regularity assumption on the coefficients $a^{ij}$ and $b^i$. However, in order to get 
an effective bound on the $(n-2)$-dimensional Hausdorff measure of the critical set, we will 
assume some additional control on the higher order derivatives of the coefficients of the PDE.

The basic ideas needed to estimate the critical sets of solutions to elliptic equations are exactly
 the same as in the harmonic case. The primary new
technical ingredient is a {\it generalized frequency} function, $\bar F(r)$
which is an almost monotone quantity, i.e., for $r$ effectively small
the function $e^{Cr}\bar F(r)$ is monotone nondecreasing; see  Theorem \ref{th_Nellmon}.  
The function $\bar F(r)$  will replace the frequency function of the harmonic case.   
It is constructed by  a generalizing a constructions of \cite{galin1,galin2}. Their function however,
 is only almost monotone for operators of divergence form on $\dR^n$ for $n\geq 3$.  
Our construction will take up most of the next subsection.  Though the proofs of many points
 involve standard techniques, we will include them for convenience and completeness.

\subsection{The Generalized Frequency Function}\label{sec_freqell}

In this section we define a generalized version of Almgren's frequency, denoted by
 $\bar F$, suitable to study the properties of solutions to \eqref{eq_Lu}.
 Even though the ideas in the construction are the same
 as in \cite{galin1,galin2} \footnote{see also the survey \cite{hanlin}}, some of the details are different.
 This allow us to prove the almost monotonicity for a wider class of operators, and in particular,
for those dealt with in this paper. 
 For the reader's convenience, we include the proof of almost monotonicity of $\bar F$.

As a first step towards the definition,
we introduce a new metric related to the coefficients $a_{ij}$, which is closely related to 
the constructions in \cite{hanlin}. For the sake of simplicity, we will occasionally use the terms 
and notations typical of Riemannian manifolds. For instance, we denote by $a_{ij}$ the elements 
of the inverse matrix of $a^{ij}$ and by $a$ the determinant of $a_{ij}$. The metric $g_{ij}$ (also denoted by $g$)
will be defined on $B_1(0)\subseteq \R^n$ and $e_{ij}$ will denote the standard Euclidean metric. 
For ease of notation, we define $B(g,x,r)$ to be the geodesic ball centered at $x$ with radius $r$
 with respect to the metric $g$.

It would seem natural to define a metric $g_{ij}=a_{ij}$ and use this metric in the definition of the
 frequency function. However, for such a metric the geodesic polar coordinates at a point $x$ are 
well defined only in a small ball centered at $x$ whose radius is not easily bounded from below with 
only Lipschitz control on the $a_{ij}$.  To avoid this problem, we define a similar but slightly different
 metric which has been introduced in \cite[eq. (2.6)]{toc}, and later used also in \cite{galin1,galin2}; see
 also the nice survey paper \cite[Section 3.2]{hanlin}. In these papers, the authors use this metric to define
 a frequency function which turns out to be almost monotone at small scales for elliptic equations in 
divergence form on $\R^n$ with $n\geq 3$, and only bounded at small enough scales for more general equations.

We will introduce a modified frequency function which we will prove to be almost monotone at small
 scales for all solutions of equation \eqref{eq_Lu}, with neither a restriction on the dimension $n$,
 nor a divergence form assumption.

To begin with, we recall from \cite{toc}, the definition and some properties of the new metric $g_{ij}$.
  Fix an origin $\bar x$, and define the function $r^2$ on the Euclidean ball $B_1(0)$ by
\begin{gather}
 r^2=r^2(\bar x,x)=a_{ij}(\bar x) (x-\bar x)^i (x-\bar x)^j\, ,
\end{gather}
where $x=x^ie_i$ is the usual decomposition in the canonical basis of $\R^n$. Note that the level sets 
of $r$ are Euclidean ellipsoids centered at $\bar x$, and the assumptions on the coefficients $a_{ij}$ lead
 to the estimate
\begin{gather}
 \lambda^{-1} \abs {x-\bar x} ^2 \leq r^2(\bar x,x)\leq \lambda \abs {x-\bar x}^2\, .
\end{gather}
\begin{proposition}\label{prop_gij}
 With the definitions above, set
\begin{gather}
 \eta(\bar x,x)={a^{kl}(x)\frac{\partial r(\bar x,x)}{\partial x^k}\frac{\partial r(\bar x,x)}{\partial x^l}}
={a^{kl}(x)\frac{a_{ks}(\bar x)a_{lt}(\bar x)(x-\bar x)^s(x-\bar x)^t}{r^2}}\, ,\\
g_{ij}(\bar x,x)=\eta(\bar x,x)a_{ij}(x)\, .
\end{gather}
Then for each $\bar x \in B_{1}(0)$, the geodesic distance $d_{\bar x}(\bar x,x)$ in the metric 
$g_{ij}(\bar x,x)$ is equal to $r(\bar x,x)$.  In particular,  geodesic polar coordinates with respect to $\bar x$ are 
well-defined on the Euclidean ball of radius $\lambda^{-1/2} (1-\abs {\bar x})$. Moreover in
 these coordinates the metric assumes the form
\begin{gather}
 g_{ij}(\bar x, (r,\theta))=dr^2+ r^2 b_{st}(\bar x, (r,\theta))d\theta^s d\theta^t\, ,
\end{gather}
where the $b_{st}(\bar x,r,\theta)$ can be extended to Lipschitz functions in 
$[0,\lambda^{-1/2}(1-\abs {\bar x})]\times \partial B_1$ with
\begin{gather}\label{eq_bst}
 \abs{\frac{\partial b_{st}}{\partial r}}\leq C(\lambda)\, ,
\end{gather}
and $b_{st}(\bar x,0,\theta)$ is the standard Euclidean metric on $\partial B_1$.
\end{proposition}
\begin{remark}\label{rem_lip}
 For the time being, let $\bar x=0$ be fixed. As seen in the proposition, if $a^{ij}$ is 
Lipschitz, then so is also the metric $g_{ij}$. However, if the coefficients $a^{ij}$ are assumed
 to have higher regularity, for example $C^1$ or $C^m$, it easily seen that $g_{ij}$ is of higher 
regularity away from the origin. But at the origin, in general, $g_{ij}$  is only Lipschitz.
\end{remark}

Before giving  the formula for the generalized frequency, we rewrite equation \eqref{eq_Lu} in a 
Riemannian form with respect to the metric $g_{ij}$. Using the Riemannian scalar product and Laplace operator, 
relation \eqref{eq_Lu} is equivalent to
\begin{gather}
 \Delta_g (u)=\ps{B\,}{\,\nabla u}_{g}\, ,
\end{gather}
where $B$ is the vector field which in the standard Euclidean coordinates\! has components
\begin{gather}\label{eq_B}
 B_i = -\eta^{-1} b_i +\frac{\partial}{\partial x^i}\log\ton{g^{1/2} \eta^{-1}}\, .
\end{gather}
Given conditions \eqref{e:coefficient_estimates}, it is easy to prove the bound 
$$
\ps{B\, }{\, B}_g =\abs B^2 _g\leq C(\lambda)\, .
$$

Now we are ready to define the generalized frequency function for a (weak) solution 
$u$ to \eqref{eq_Lu}. For convenience of notation, we will denote this new frequency $\bar F$.
\begin{definition}\label{deph_LN}
For a solution $u$ to equation \eqref{eq_Lu}, for each $\bar x\in B_1(0)$ and
 $r\leq \lambda^{-1/2}(1-\abs {\bar x})$, define
\begin{gather}
 D(u,\bar x,g,r)=\int_{B(g(\bar x),\bar x,r)}\norm{\nabla u}_{g(\bar x)}^2dV_{g(\bar x)}
=\int_{r(\bar x,x)\leq r} \eta^{-1}(\bar x,x) a^{ij}(x) \partial_i u \partial_j u\sqrt{\eta^n(\bar x,x)a(x)}dx\, .\\
\!\!\!\!\! \!\!\!\!\! \!\!\!\!\! \!\!\!\!\!\!\!\!\!\! \!\!\!\!\! \!\!\!\!\! \!\!\!\!\!\!\!\!\!\! \!\!\!\!\!\!\!\! \!\!\!\!\!
\!\!\!\!
 I(u,\bar x,g,r)=\int_{B(g(\bar x),\bar x,r)}\norm{\nabla u}_{g(\bar x)}^2
+ (u-u(\bar x))\Delta_{g(\bar x)} (u )dV_{g(\bar x)}=\\
\notag
\!\!\!\!\!\!\!\!\!\!\!=\int_{r(\bar x,x)\leq r} \norm{\nabla u}_{g(\bar x)}^
2+ (u-u(\bar x))\ps{B\, }{\, \nabla u}_{g(\bar x)} dV_{g(\bar x)}\, .\\
\!\! \!\!\!H(u,\bar x,g,r)=\int_{\partial B(g(\bar x) ,\bar x,r)} 
\qua{u-u(\bar x)}^2 dS_{g(\bar x)}=r^{n-1}\int_{\partial B_1} \qua{u(r,\theta)-u(\bar x)}^2
\sqrt{b(\bar x,r,\theta)}d\theta\, .\\
\!\!\!\!\! \!\!\!\!\! \!\!\!\!\! \!\!\!\!\!\!\!\!\!\! \!\!\!\!\! \!\!\!\!\! \!\!\!\!\! \!\!\!\!\! \!\!\!\!\! \!\!\!\!\!\!\!\!\!\! \!\!\!\!\! \!\!\!\!\!
 \!\!\!\!\! \!\!\!\!\! \!\!\!\!\! \!\!\!\!\!\!\!\!\!\! \!\!\!\!\! \!\!\!\!\! 
\!\!\!\!\! \!\!\!\!\! \!\!\!\!\! \!\!\!\!\!\!\!\!\!\! \!\!\!\!\! \!\!\!\!\! \!\!
\bar F(u,\bar x,g,r) 
=\frac{rI(u,\bar x,g,r)}{H(u,\bar x,g,r)}\, .
\end{gather}

Note that, by elliptic regularity, $\bar F$ is a locally Lipschitz function for $r>0$. Moreover,
since $u$ is not constant, by  unique continuation  and the maximum principle, $H(r)>0$ for 
all positive $r$. So $\bar F$ is well-defined. Note also that if the operator $\L$ in \eqref{eq_Lu}
 is the usual Laplace operator, then it is easily seen that $\bar F(u,x,g,r)= \bar N^u(x,r)$.
\end{definition}
For $t$ sufficiently small, we can bound $D$ in terms of $I$ and vice versa.
Moreover, by using the Poincaré inequality, we can  bound $\bar F$ away from zero.
\begin{proposition}\label{prop_ID}
 Fix $u$, $x$ and the relative metric $g$. There exists a constant $C(\lambda)$ and 
$r_0=r_0(n,\lambda)>0$ such that for all admissible $r$, 
$$
I(r)\leq C D(r)\, ,
$$
while for $r\leq r_0$, 
$$
D(r)\leq C I(r)\, .
$$
 Moreover, there exits $c(n,\lambda)>0$ for which 
$$
\bar F(r)\geq c(n,\lambda)\, ,
$$ for all $r\leq r_0$.
\end{proposition}
\begin{proof}
 Assume for simplicity that $x=0$ and $u(0)=0$. By definition, we have
\begin{gather}
I(r)= D(r) + \int_{B(r)} u \ps{B\, }{\, \nabla u} dV\, .
\end{gather}
Using H\"older  and Poincaré's inequalities, it is easy to see that there exists a constant 
$C(\lambda)$ for which
\begin{gather}
\abs{ \int_{B(r)} u \ps{B\, }{\, \nabla u} dV }\leq C(\lambda) \sqrt{\int_{B(r)} u^2 dV }
 \cdot D(r)^{1/2}\leq C(\lambda) r D(r)\, .
\end{gather}
Thus, the estimates follow easily.

 For the lower bound on $\bar F$, note that
\begin{gather}
 \int_{\partial B(r)} u^2 dS = \frac 1 r \int_{\partial B(r)} u^2 \ps{\v\,\, }{\, \,\hat n} dS 
= \frac 1 r \int_{B(r)} 2u\ps{\nabla u\, }{\, \vec x} dV + \frac{1}r\int_{B(r)} u^2 \operatorname{div}(\v) dV\, ,
\end{gather}
where $\vec v$ is the Lipschitz vector field $r\partial_r$. By conditions \eqref{e:coefficient_estimates}, 
$\dive\v\leq C(n,\lambda)$, and a simple application of Poincaré's inequality leads to
\begin{gather}
 H(r)\leq c^{-1}(n,\lambda) r D(r) \leq c^{-1}(n,\lambda) r I(r)\, .
\end{gather}
\end{proof}
The frequency function $\bar F$ has invariance properties similar to those which hold for 
harmonic functions. For instance, it is invariant under blow-ups, as long as they are
 redefined in a geodesic sense. The following lemma is the counterpart of Lemma \ref{l_blowup}.
\begin{lemma}\label{lemma_Ninvell}
 Let $u$ be a nonconstant solution to \eqref{eq_Lu}. Fix $x\in B_1(0)$ and the relative 
metric $g_{ij}$ as in Proposition \ref{prop_gij}. Consider the blow up given in geodesic
 polar coordinates centered at $x$ by $(r,\theta)\to (tr,\theta)$. If we define $w(r,\theta)
=\alpha u(tr,\theta)+\beta$ and $g^t_{ij}(r,\theta)=g_{ij}(tr,\theta)$,  then
\begin{gather}
 \bar F(u,x,g,r)=\bar F(w,x,g^t,t^{-1}r)\, .
\end{gather}
\end{lemma}
\paragraph{Blow-up function $U_{x,r}$}
Here we define two auxiliary functions $T_{x,r}$ and $U_{x,r}$ which are generalizations of the
 blow-up function $T_{x,r}$ for harmonic functions. 

Using the geodesic blow-up given in the previous lemma, we introduce the function $U_{x,t}u(y)$
as follows.
\begin{definition}\label{deph_Uell}
 We define
\begin{gather}\label{eq_Tt}
 U_{x,t} u(r,\theta)\equiv \frac{u(tr,\theta)-u(0)}{\ton{\fint_{\partial B(g(0),0,t)} 
[u(r,\theta)-u(0)]^2dS(g)}^{1/2} }\quad \quad  U_{x,t}u(0)=0\, .
\end{gather}
\end{definition}

Note that elliptic regularity ensures that for all $t$, $U_{x,t}u\in W^{2,p}(B_1(0))\cap C^{1,\alpha}(B_1(0))$. 
Moreover, $U_{x,t}$ is normalized in the sense that:
\begin{gather}
 \fint_{\partial B(g(x)^t,0,1)} \abs{U_{x,t}}^2 dS(g(x)^t)=1\, .
\end{gather}
Using a simple change of variables, it is easy to see that $U_t$ satisfies (in the weak sense) the equation
\begin{gather}\label{eq_dt}
 \Delta_{g(x)^t} U_{x,t} = t \ps{B\, }{\, \nabla U_{x,t} u}_{g(x)^t}\, ,
\end{gather}
where $B$ is defined by equation \eqref{eq_B}.

\paragraph{Blow-up function $T_{x,r}$}
For a fixed $x$, let $q^{ij}(x)$ be the square root of the matrix $a^{ij}(x)$, and define
 the linear operator $Q_{x}$ by
\begin{gather}\label{eq_Q}
 Q_{x} (y) = q_{ij}(y-x)^i e_j\, .
\end{gather}
It is evident that, independently of $x$,  $Q_x$ is a bi-Lipschitz equivalence from $\R^n$ 
to itself with Lipschitz constant $(1+\lambda)^{1/2}$. Moreover, note that the ellipsoid
 $\|Q_x(y)\|\leq r$ 
 is exactly the geodesic ball $B(g(x),x,r)$, where $g(x)$ is the metric
 introduced in Proposition \ref{prop_gij}.

\begin{definition}\label{deph_Tell}
Define the function $T_{x,t}:B_1(0)\to \R$ by
\begin{gather}
 T_{x,t}(y) = \frac{u(x + t Q_{x}^{-1}(y))-u(x)}{\ton{\int_{\partial B_1} [u(x + t Q_{x}^{-1}(y))-u(x)]^2  
 dS}^{1/2}}\, .
\end{gather}
\end{definition}
Using a simple change of variables, it is easy to see that the function $T$ satisfies an 
elliptic PDE of the form:
\begin{gather}
 \tilde \L(u)=\partial_i\ton{\tilde a^{ij} \partial _j T} + \tilde b^i \partial_i T =0\, ,
\end{gather}
with $\tilde a^{ij}(x)=\delta^{ij}$. Moreover, as long as $t\leq 1$, condition \eqref{eq_cond} implies a similar estimate for the coefficients $\tilde a^{ij}, \ \tilde b^i$:
\begin{gather}\label{eq_U}
 \norm{\tilde a^{ij}}_{C^M(B_1)}, \ \norm{\tilde b^{i}}_{C^M(B_1)}\leq C(n,\lambda, L)\, .
\end{gather}
Thus, on $B_1(0)$ we have uniform elliptic estimates on $T_{x,t} u(y)$ for $x\in B_{1/2}(0)$ and $t\leq (1+\lambda)^{-1}/3$.

Note in addition that as $t$ converges to $0$, $U_{x,t}$ converges to $T_{x,t}$ in $C^{0,1}(B_1(0))$.

\paragraph{Almost monotonicity}
By an argument that is philosophically identical to the one for harmonic functions, although technically more complicated, we show that this modified frequency is almost monotone in the following sense.
\begin{theorem}\label{th_Nellmon}
 Let $u:B_1(0)\to \R$ be a nonconstant solution to equation \eqref{eq_Lu} with \eqref{e:coefficient_estimates} and let $x\in B_{1/2}(0)$. Then there exists a positive $r_0=r_0(\lambda)$ and a constant $C=C(n,\lambda)$ such that
\begin{gather}
 e^{C r} \bar F(r)\equiv  e^{C r} \bar F(u,x,g(x),r)
\end{gather}
is monotone nondecreasing on $(0,r_0)$.
\end{theorem}
\begin{proof}
By a standard $C^{1,\alpha}$ density argument, we can assume that $a^{ij}$ and $b^i$ are smooth. Indeed, there exists a sequence of smooth solutions to elliptic pdes with smooth coefficients that converge in the $C^{1,\alpha}$ sense to $u$.
Moreover, for simplicity we assume $x=0$ and $u(0)=0$. We will prove that, for $r\in (0, r_0)$:
\begin{gather}\label{eq_dN}
 \frac{\bar F'(r)}{\bar F(r)}\geq -C(n,\lambda)\, .
\end{gather}
Define $U_{t} u=U_{0,t} u$ as in \eqref{eq_Tt}. Using Lemma \ref{lemma_Ninvell}, the last statement is equivalent to
\begin{gather}
 \frac{\bar F'_t(1)}{\bar F_t(1)}\equiv \frac{\bar F'(U_tu,g^t,0,1)}{\bar F(U_tu,g^t,0,1)}\geq -C(n,\lambda)t\, .
\end{gather}

For the moment, fix $t$ and set $U=U_t u$. We begin by computing the derivative of $H$.  We have,
\begin{gather}
 H(r)=H(U,g^t,0,r)=r^{n-1}\int_{\partial B_1} U^2(r,\theta) \sqrt{b(tr,\theta)}d\theta\, ,\\
H'|_{r=1}=(n-1) H(1) +  2\int_{\partial B_1} U\ps{\nabla U\, }{\, \nabla r}\sqrt{b(t,\theta)}d\theta+\notag \int_{\partial B_1}\ton{\frac t 2 \left.\frac{\partial \log(b)}{\partial r}}\right\vert_{(tr,\theta)} U^2(1,\theta) \sqrt{b(t,\theta)}d\theta\, .
\end{gather}
By using equation \eqref{eq_bst}, we obtain the estimate 
\begin{gather}\label{eq_N1}
 \abs{H'(1)- (n-1)  H(1) -  2\int_{\partial B(g^t,0,1)} U U_ndS(g^t)}\leq C(n,\lambda) t \ H(1)\, ,
\end{gather}
where $U_n=\ps{\nabla U}{\partial_r}$ is the normal derivative of $U$ on $\partial B(g^t,0,r)$. As for the derivative of $I$, we split it into two parts:
\begin{equation}\label{eq_N2}
\begin{aligned}
 I'=\frac{d}{dr}I(U,g^t,r)&=\int_{\partial B(g^t,0,r)}\ton{\norm{\nabla U}_{g^t}^2+ U\Delta_{g^t} (U )}dS(g^t)\\
&= \int_{\partial B(g^t,0,r)}\norm{\nabla U}_{g^t}^2dS(g^t)+\int_{\partial B(g^t,0,r)} U\Delta_{g^t} (U )dS(g^t)\\
&\equiv I'_\alpha + I'_\beta\, .
\end{aligned}
\end{equation}
Using geodesic polar coordinates relative to $g^t$, set $\v=r \nabla r$. By the divergence theorem we get 
\begin{equation}
\begin{aligned}
I'_\alpha &= \frac 1 r \int_{\partial B(g^t,0,r)} \norm{\nabla U}_{g^t}^2\ps{\vec v\, }{\, r^{-1}\vec v}dS(g^t)
=\frac 1 r \int_{B(g^t,0,r)} \dive{\norm{\nabla U}_{g^t}^2\vec v }dV(g^t)\\
&=\frac 1 r \int_{B(g^t,0,r)}\norm{\nabla U}_{g^t}^2 \dive{\vec v }dV(g^t)
+ \frac 2 r \int_{B(g^t,0,r)} \nabla^i\nabla^j U\, \nabla_i U\,  \vec v _j\, dV(g^t)\\
 &=\frac 1 r \int_{B(g^t,0,r)}\norm{\nabla U}_{g^t}^2 \dive{\vec v }dV(g^t)
+\frac 2 r \int_{B(g^t,0,r)} \ps{\nabla \ps{\nabla U\,}{\, \vec v}\, }{\, \nabla U} dV(g^t)
 - \frac 2 r \int_{B(g^t,0,r)} \nabla^j U\nabla_i U \ton{\nabla^i\vec v }_j dV(g^t)\\
&=\frac 1 r \int_{B(g^t,0,r)}\norm{\nabla U}_{g^t}^2 \dive{\vec v }dV(g^t) 
+ 2 \int_{\partial B(g^t,0,r)} \ton{U_n}^2 dS(g^t) \\
&{}\qquad-\frac 2 r\int_{B(g^t,0,r)} t\ps{\nabla U\, }{\, \vec v}\ps{B}{\nabla U} dV(g^t)
-  \frac 2 r \int_{B(g^t,0,r)} \nabla^j U\nabla_i U \ton{\nabla^i\vec v }_j dV(g^t)\, .
\end{aligned}
\end{equation}
Using geodesic polar coordinates, it is easy to see that
\begin{gather}
 \left.\abs{\ton{\nabla^i \vec v}_j -\delta^i_j}\right\vert_{(r,\theta)}\leq r t C(\lambda)\, .
\end{gather}
Therefore,  we have the estimate
\begin{gather}
 \abs{I'_\alpha(1)- (n-2) D(1) - 2 \int_{\partial B(g^t,0,1)} \ton{U_n}^2 dS(g^t)}\leq t C(n,\lambda) D(1)\, .
\end{gather}
Using Proposition \ref{prop_ID} we conclude that for $t\leq r_0=r_0(\lambda)$,
\begin{gather}\label{eq_I}
 \abs{I'_\alpha(1)-(n-2) I(1)- 2 \int_{\partial B(g^t,0,1)}\ton{U_n}^2 dS(g^t)}\leq t C(n,\lambda) I(1)\, .
\end{gather}

To estimate $I'_\beta$, we use the divergence theorem to write
\begin{gather}
 I(r)=\int_{\partial B(g^t,0,r)} U U_n dS(g^t)\, .
\end{gather}
Note that for $tr\leq r_0$, $I(r)>0$.
 From Cauchy's inequality and Proposition \ref{prop_ID}, we get
\begin{gather}
 \notag  I^2(r)\leq H(r)\int_{\partial B(g^t,0,r)} U_n^2 dS(g^t)
\leq \frac{rI(r)}{c(n,\lambda)}\int_{\partial B(g^t,0,r)} U_n^2 dS(g^t)\, ,\\
\!\!\!\!\!\!\!\!\!\!\!\!\!\!\!\!\!\!\!\!\!\!\!\!\!\!\!\!\!\!\!\!\!\!\!\!\!\!\!\!\!\!\!\!\!\!\!\!\!\!\!\!\!\!\!\!\!\!
 I(r)\leq \frac{r}{c(n,\lambda)}\int_{\partial B(g^t,0,r)} U_n^2 dS(g^t)\, ,
\end{gather}
and so, using equation \eqref{eq_I}, we get
\begin{gather}\label{eq_cau}
 \int_{\partial B(g^t,0,1)} \norm{\nabla U}^2_{g^t} dS(g^t)= I'_\alpha(1) \leq C(n,\lambda)
 \int_{\partial B(g^t,0,1)} U_n^2 dS(g^t)\, .
\end{gather}
 Following \cite[pag 56]{hanlin}, we divide the rest of the proof in two cases.
\paragraph{Case 1.}
Suppose 
\begin{gather}\label{eq_case1}
 \int_{\partial B(g^t,0,1)} U^2 dS(g^t) \ \int_{\partial B(g^t,0,1)} U_n^2 dS(g^t)\leq 2 
\ton{\int_{\partial B(g^t,0,1)} UU_n dS(g^t)}^2= 2I^2(1)\, .
\end{gather}
In this case, using Cauchy's inequality and \eqref{eq_cau}, we have the estimate
\begin{gather}\label{eq_N3}
\abs{ I'_\beta(1)}= \abs{\int_{\partial B(g^t,0,1)} t U \ps{B}{\nabla U}dS(g^t)} \leq tC(n,\lambda) I(1)\, .
\end{gather}
So, from equations \eqref{eq_N1}, \eqref{eq_N2}, \eqref{eq_I} and \eqref{eq_N3}, we get for $t\leq r_0$,
\begin{gather}
\notag \frac{\bar F_t'(1)}{\bar F_t(1)}= 1 + \frac{I'(1)}{I(1)} - \frac{H'(1)}{H(1)}
\geq 0 + 2\ton{\frac{\int_{\partial B(g^t,0,1)}  U_n^2dS(g^t)}{\int_{\partial B(g^t,0,1)}  UU_ndS(g^t)}
-\frac{\int_{\partial B(g^t,0,1)}  UU_ndS(g^t)}{\int_{\partial B(g^t,0,1)}  U^2dS(g^t)}} - tC(n,\lambda)\geq 
- tC(n,\lambda)\, ,
\end{gather}
where the last inequality comes from a simple application of Cauchy's inequality.
\paragraph{Case 2.}
To complete the proof, suppose
\begin{gather}\label{eq_case2}
 \int_{\partial B(g^t,0,1)} U^2 dS(g^t) \ \int_{\partial B(g^t,0,1)} U_n^2 dS(g^t)
>2 \ton{\int_{\partial B(g^t,0,1)} UU_n dS(g^t)}^2= 2I^2(1)\, .
\end{gather}
Then we have the following estimate for estimate $I'_\beta$.
\begin{gather}
\abs{ I'_\beta(1)}= \abs{\int_{\partial B(g^t,0,1)} t U \ps{B}{\nabla U}dS(g^t)}
 \leq t \ton{\int_{\partial B(g^t,0,1)}  U^2dS(g^t)\int_{\partial B(g^t,0,1)} \norm{\nabla U}_{g^t}^2dS(g^t)}^{1/2}
\leq \\
\notag \!\!\!\!\!\!\!\!\!\!\!\!\!\!\!\!\!\!\!\!\!\!\!\!\!\!\!\!\!\!\!\!\!\!\!\!\!\!\!\!\!\!\!\!\!\!\!\!\!\!\!\!\!\!\!\!
\!\!\!\!\!\!\!\!\!\!
\leq C(n,\lambda)t\ton{\int_{\partial B(g^t,0,1)}  U^2dS(g^t)\int_{\partial B(g^t,0,1)} U_n^2dS(g^t)}^{1/2}\, .
\end{gather}
Applying Young's inequality with the right constant and Proposition \ref{prop_ID}, we obtain that for $t\leq r_0$,
\begin{gather}\label{eq_N4}
\abs{ I'_\beta(1)}\leq \int_{\partial B(g^t,0,1)} U_n^2dS(g^t) + C(n,\lambda)t^2 \int_{\partial B(g^t,0,1)}  U^2dS(g^t)
 \leq \int_{\partial B(g^t,0,1)} U_n^2dS(g^t) + C(n,\lambda)t^2 I(1)\, .
\end{gather}
Using equations \eqref{eq_N1}, \eqref{eq_N2}, \eqref{eq_I} and \eqref{eq_N4}, we get for $t\leq r_0$,
\begin{gather}
\frac{\bar F_t'(1)}{\bar F_t(1)}= 1 + \frac{I'(1)}{I(1)} - \frac{H'(1)}{H(1)}\geq 0 
+ {\frac{\int_{\partial B(g^t,0,1)}  U_n^2dS(g^t)}{\int_{\partial B(g^t,0,1)}  UU_ndS(g^t)
}-\frac{2\int_{\partial B(g^t,0,1)}  UU_ndS(g^t)}{\int_{\partial B(g^t,0,1)}  U^2dS(g^t)}} - tC(n,\lambda)\geq - tC(n,\lambda)\, ,
\end{gather}
where the last inequality follows directly from the assumption \eqref{eq_case2}.
\end{proof}

For the proof of Theorem \ref{th_main_proof}, Lemma \ref{lemma_strong2bar} is crucial.
 It states that a bound on $\bar N^u(0,1)$ gives a bound also on $\bar N^u(x,r)$, for 
well-chosen $x$ and $r$. A similar statement holds for solutions to \eqref{eq_Lu}. 
However this statement is valid only for $r\leq r_0(n,\lambda,\Lambda)$.
\begin{lemma}\label{lemma_prestrongell}
 There exists $r_0=r_0(n,\lambda,\Lambda)$ and $C=C(n,\lambda,\Lambda)$ such that if 
$u$ is a solution to \eqref{eq_Lu} with \eqref{e:coefficient_estimates} on 
$B_{{\lambda}^{-1/2}r}(0)$, $0<r\leq r_0$ and $\bar F(0,r)\leq \Lambda$, then for all $x\in B_{r/3}(0)$,
 \begin{gather}
  \bar F(x,r/3)\leq C\, .
 \end{gather}
 \end{lemma}
\begin{remark}
 Even though it might be possible to prove this lemma using doubling conditions for 
$H(r)$ and mean value theorems, it is much more convenient to set up a 
contradiction/compactness argument. Such an argument does not give explicit
 quantitative control on the constants $C$ and $r_0$. Rather, it only proves their existence. 
 For our purposes, this is sufficient.
\end{remark}
\begin{proof}
 Assume, by contradiction, that there exists
a sequence of solutions $u_i$ to $\L_i(u_i)=0$, where the operators $\L_i$ satisfy 
conditions \eqref{e:coefficient_estimates}. Assume also that $\bar F(u_i,0,,g^i(0),i^{-1})
\leq \Lambda$, but for some $x_i\in B_{i^{-1}/3}(0)$, $\bar F(u_i,x_i,g^i(x_i),i^{-1}/3)\geq i$. 
For each operator $\L_i$, consider the associated metric $g$ at the origin and define $g^i(r,\theta)=g(i^{-1}r,\theta)$.
 An easy consequence of the conditions \eqref{e:coefficient_estimates} is that $g^i(r,\theta)$ converges
 in the Lipschitz sense on $B_1(0)$ to the Euclidean metric.

For simplicity, set $U_i(r,\theta)= U_{0,i^{-1}} u_i(r,\theta)$, where the latter is defined in equation \eqref{eq_Tt}.

The bound on the frequency $\bar F$ together with Lemma \ref{prop_ID} implies that, for $i$ large enough,
\begin{gather}
 \int_{B_1}\abs{\nabla U_i}^2 dV \leq \lambda^{\frac{n-2}2} \int_{B_1(0)} \norm{\nabla U_i}^2_{g^i}dV(g^i)
\leq C(n,\lambda)\bar F (0,i^{-1}) \leq C(n,\lambda)\Lambda\, .
\end{gather}
Since $U_i(0)=0$, $U_i$ have uniform bound in the $W^{1,2}(B_1(0))$ norm and, by elliptic estimates, 
also in the $C^{1,\alpha}(B_{2/3})$ norm.

Consider a subsequence $U_i$ which converges in the weak $W^{1,2}$ sense to some $U$, and a 
subsequence of $x_i$ converging to some $x\in \overline B_{1/3}$. It is easy to see that $U$ is a 
nonconstant harmonic function, and, by the convergence properties of the sequence $U_i$, we also have
\begin{gather}
 \lim_{i\to \infty} \bar F(U_i,0,g^i(0),1) = \bar F(U,0,e,1)=\bar N^U (0,1)\,  ,\\
 \lim_{i\to \infty} \bar F(U_i,x_i,g^i(x_i),1/3) = \bar F(U,x,e,1/3)=\bar N^U (x,1/3)\, .
\end{gather}
Recall that $e$ is the standard Euclidean metric on $\R^n$. The contradiction is a consequence of Lemma \ref{lemma_strong2bar}.
\end{proof}

With a standard compactness argument, we can turn the previous lemma into the following statement.
\begin{lemma}\label{lemma_Nec}
 Let $u:B_1(0)\to \R$ be a nonconstant solution to \eqref{eq_Lu} with \eqref{e:coefficient_estimates}. 
Then there exist constants $r_1(n,\lambda,\Lambda)$ and $C(n,\lambda,\Lambda)$ such that
 if $\bar N^u(0,1)\leq \Lambda$, then for all $x\in B_{1/2}(0)$ and $r\leq r_1$
\begin{gather}
 \bar F(u,x,r)\leq C(n,\lambda,\Lambda)\, .
\end{gather}
\end{lemma}

\subsection{The Frequency Decomposition and Cone-Splitting}
Similar properties to the one proved for harmonic function in Section \ref{sec_spl} are available also
 for solutions to \eqref{eq_Lu}, although it is necessary to restrict the result to scale smaller than 
some $r_0(n,\lambda,\Lambda)$. In some sense, the smaller scales the closer the solutions to \eqref{eq_Lu} 
are to harmonic functions, so if we choose the scale small enough we can replace ``harmonic'' with ``elliptic''
 without changing the final result.

The proofs of the following theorems are obtained using arguments similar to the proof of 
Proposition \ref{lemma_prestrongell} and contradiction/compactness arguments like the ones
 in Section \ref{sec_spl}. For this reason, we omit them.

\begin{theorem}\label{th_N2homell}
 Fix $\eta>0$ and $0\leq \gamma <1$, and let $u:B_1(0)\to \R$ be a nonconstant solution 
to \eqref{eq_Lu} with $\bar N^u(0,1)\leq \Lambda$. Then there exist positive 
$\epsilon=\epsilon(n,\lambda,\eta,\gamma,\Lambda)$ and 
$r_2=r_2(n,\lambda,\eta,\gamma,\Lambda) $ such that if $r\leq r_2$ and
\begin{gather}
\bar F(0,r)-\bar F(0,\gamma r)< \epsilon\, ,
\end{gather}
then $u$ is $(0,\eta,r,0)$-symmetric.

\end{theorem}

In a similar way, we can also prove a generalization of Corollary \ref{cor_symsum}:
\begin{corollary}\label{cor_symsumell}
 Fix  $\eta>0,\, \tau>0$, $0<\chi\leq 1$ and $k\leq n-2$. There exist 
$\epsilon(\lambda,\eta,\tau,\chi,\Lambda)$ and $r_3=r_3(\lambda,\eta,\tau,\chi,\Lambda)$ 
with the following property. Assume $u$ solves \eqref{eq_Lu} with $\bar N^u(0,1)\leq \Lambda$
 and for some $x\in B_{1/2}(0)$ we have
\begin{enumerate}
 \item $u$ is $(0,\epsilon,\chi r_3,x)$-symmetric,
 \item for every affine subspace $V$ passing through $x$ of dimension $\leq k$,
 there exists $y\in B_{\chi r_3}(x)\setminus B_{\tau}V$ such that $u$ is $(0,\epsilon,\chi r_3,y)$-symmetric.
\end{enumerate}
Then $u$ is $(k+1,\epsilon, r_3,x)$-symmetric.
\end{corollary}

By \eqref{e:coefficient_estimates}, we have uniform $C^{1,\alpha}$ estimates on the solutions to
 \eqref{eq_Lu} (see \cite{GT} for details). For this reason, it is straightforward to prove the following 
proposition, which is a generalization of Proposition \eqref{prop_qcrit}.
\begin{proposition}\label{prop_qcritell}
 Let $u:B_1(0)\to \R$ be a solution to \eqref{eq_Lu} with \eqref{e:coefficient_estimates} such that 
$\bar F^u(0,1)\leq \Lambda$. For every $\epsilon>0$ and $0\leq \alpha <1$, there exists positive 
$\bar \eta$ and $r_0$ depending on $(n,\epsilon,\alpha,\lambda,\Lambda)$ such that if for some 
$x\in B_{1/2}(0)$ and $r\leq r_0$ $u$ is $(n-1,\bar \eta,r,x)$-symmetric, then
 \begin{gather}
  \norm{U_{x,r}u -L}_{C^{1,\alpha}(B_{1/2})}\leq \epsilon\, ,
 \end{gather}
where $L$ denotes a linear function satisfying $\fint_{\partial B_1} \abs L^2 dS =1$. In particular, 
by choosing $\alpha=0$ and $\epsilon$ sufficiently small,  there exist positive $\eta$ and $r_0$ 
depending on $n,\lambda,\Lambda$, such that if $u$ $(n-1,\eta,r,x)$-symmetric, then $u$ does 
not have critical points in $B_{r/2}(x)$.
\end{proposition}

\subsection{Minkowski Type Estimates and the Proof of 
Theorem \ref{t:quantstrat}}
Now we are ready to prove Theorem \ref{t:quantstrat}. As in the harmonic case, we prove the
 theorem only for some $r=\gamma^j$ for a suitable value of $0<\gamma<1$ and every $j$, the 
general case follows easily from this. For the reader's convenience, here we restate the theorem in this context.
\begin{theorem}\label{th_main_ell}
Let $u:B_1(0)\to\dR$ be a solution to \eqref{eq_Lu} with \eqref{e:coefficient_estimates} and
 such $\bar N^u(0,1)\leq \Lambda$. Then for some $0<\gamma(n,\eta,\lambda,\Lambda)<1$,
 for every $j\in \N$, $\eta>0$ and $k\leq n-2$ we have 
\begin{align}
{\rm Vol}\ton{B_{\gamma^j}(\cS^k_{\eta,\gamma^j})\cap B_{1/2}(0)}
\leq C(n,\lambda,\Lambda,\eta) \ton{\gamma^j}^{n-k-\eta}\, .
\end{align}
\end{theorem}
\begin{proof}
Since the proof is almost identical to that of Theorem \ref{th_main_proof}, 
 we simply mention how to adapt the proof from the harmonic case.

Fix  $\eta>0$ and let $\gamma=c_0^{-2/\eta}<1$, $\chi=\gamma$. Let $\tau>0$. 
Take $r_0$ to be the minimum of $r_1$ given by Lemma \ref{lemma_Nec}, $r_2$ given 
by Corollary \ref{th_N2homell} and 
let $r_3$ be given by Corollary \ref{cor_symsumell}. Then, if $i$ is large enough so 
that $\gamma^i\leq r_0$, then the same proof as in the harmonic case applies also to this 
more general case with Lemma \ref{lemma_strong2bar} replaced by Lemma \ref{lemma_Nec}, 
Theorem \ref{th_N2hom} by \ref{th_N2homell} and Corollary \ref{cor_symsum} by Corollary \ref{cor_symsumell}.

Note that $\gamma^i> r_0$ for only a finite number of exponents $i$, and 
that the number of such exponents  is bounded by a uniform constant $D'=D'(n,\lambda,\eta,\Lambda)$. 
 Finally, even though in the elliptic case
$\bar F$ not monotone, but rather, only almost monotone, it is straightforward to see that an estimate of the form
 given in equation \eqref{eq_estK} still holds.
\end{proof}

\begin{remark}
The main application for this theorem is the volume estimate on the tubular neighborhoods of the 
critical set (Theorem \ref{t:crit_lip}). As in the harmonic case, this theorem is a simple corollary of 
Theorem \ref{t:quantstrat} and Proposition \ref{prop_qcritell}.
\end{remark}

%%%%%%%%%%%%%%%%%%%%%%%%%%%%%%%%%%%%%%%%%%%%%%%%%%%%%%%%%%%%%%%%%%%%%%%%%%%%%%%%%%%%%%%%%%

\subsection[Estimates on (n-2)-dimensional Hausdorff Measure, for  Solutions of Elliptic Equations]{Estimates 
on $(n-2)$-dimensional Hausdorff Measure, for  Solutions of
 Elliptic Equations}
As for the Minkowski type estimates, it is also possible to generalize the effective estimates for the critical set involving
$(n-2)$-dimensional Hausdorff measure, to solutions to elliptic equations of the form \eqref{eq_Lu}. However for this
 estimate, we require higher order regularity assumptions on the coefficients $a^{ij}$ and $b^i$.

The following lemma is the generalization of Corollary \ref{cor_ereg} for solutions to \eqref{eq_Lu}.
\begin{lemma}\label{lemma_eregell}
 Let $P$ be an $(n-2)$-symmetric homogeneous harmonic polynomial normalized with $\fint_{\partial B_1} P^2 dS =1$.
 Let $u:B_1(0)\to \R$ be a solution to \eqref{eq_Lu} with conditions \eqref{e:coefficient_estimates} and such that 
$\bar N^u(0,1)\leq \Lambda$. There exists a positive integer $M=M(n,\lambda,\Lambda)$ such that if
\begin{gather}\label{eq_cond}
 \norm{a^{ij}}_{C^M(B_1(0))}, \ \norm{b^i}_{C^M(B_1(0))} \leq L\, ,
\end{gather}
then there exist positive $C=C(n,L,\Lambda)$, $\bar r = \bar r (n,L,\Lambda)$, $\epsilon=\epsilon(n,L,\Lambda)$ 
and $\chi=\chi(n,L,\Lambda)$ such that if for some $x\in B_{1/2}(0)$ and $r\leq \bar r$ we have
\begin{gather}\label{eq_estT}
 \fint_{\partial B_1(0)} \abs{U_{x,r} u -P}^2 dS < \epsilon\, ,
\end{gather}
then for all $s\leq \chi r$,
\begin{gather}
 H^{n-2}\ton{\nabla u^{-1}(0) \cap B_{s}(x)}\leq C s^{n-2}\, .
\end{gather}
\end{lemma}
\begin{proof}
As in the harmonic case, this lemma is a corollary of Lemma \ref{lemma_n-2Ha}.
The only delicate aspect is the generalization of the elliptic estimates.

Recall that the metric $g(\bar x)$ defined in Proposition \ref{prop_gij} is only Lipschitz at the origin, 
no matter the regularity of $a^{ij}$. Thus, it is not possible to obtain bounds on the higher order 
derivatives of $U_{x,r}u$. However, we do have uniform higher order elliptic estimates on $T_{x,r}$.
As $t$ approaches zero, $U_{x,t}$ converges in the Lipschitz sense to $T_{x,t}$. So, for $t$ small 
enough, condition \eqref{eq_estT} implies
\begin{gather}\label{eq_estU}
 \fint_{\partial B_1(0)} \abs{T_{x,r} u -P}^2 dS < \epsilon\, .
\end{gather}
By a simple application of Lemma \ref{lemma_n-2Ha}
(the $\epsilon$-regularity lemma)
the conclusion follows just as in the harmonic case.
\end{proof}

\begin{remark}
Following the same scheme as in the harmonic case it is now easy to prove Theorem \ref{th_Ha} for 
solutions to \eqref{eq_Lu}.
\end{remark}

\section{The Singular Set}\label{sec_sing}
With simple modifications, the quantitative stratification technique can also be used to derive estimates 
on the singular sets of solutions to \eqref{eq_Lu2} with \eqref{e:coefficient_estimates}.

Since constant functions do not solve \eqref{eq_Lu2}, we cannot use the normalized frequency function.
 For solutions to homogeneous elliptic equations with a zero order term, we can define the generalized frequency
 function $F(x,r)$ by
\begin{gather}
F(u,\bar x,g,r)=\frac{r\int_{B(g(\bar x),\bar x,r)}\norm{\nabla u}_{g(\bar x)}^2
+ u\Delta_{g(\bar x)} (u )dV_{g(\bar x)}}{\int_{\partial B(g(\bar x),\bar x,r)} u^2 dS_{g(\bar x)}}\, .
\end{gather}
This function turns out to be almost monotone as a function of $r$ on $(0,r_0(\lambda))$ if $u(\bar x)=0$.

Once this is proved, it is not difficult to see that a theorem similar to \ref{t:crit_lip}
 holds for solutions to this kind of elliptic equation, although in this case, the $(n-2+\eta)$-Minkowski 
type estimate holds
on  the \textit{singular} set, not the \textit{critical} set.
\begin{theorem}\label{th_si1}
Let $u:B_1(0)\to \R$ be a solution to \eqref{eq_Lu} with \eqref{e:coefficient_estimates}
 and such that $\bar N^u(0,1)\leq \Lambda$. For every $\eta>0$, there exists a positive 
$C=C(n,\lambda,\Lambda,\eta)$ such that
\begin{gather}
 {\rm Vol}\qua{B_r \ton{\Cr(u)\cap u^{-1}(0)}\cap B_{1/2}(0)}\leq C r^{2-\eta}\, .
\end{gather}
\end{theorem}

We also point our that the effective $(n-2)$-dimensional Hausdorff measure estimate 
is easily generalized to the singular set in this context, although even in this case, we need 
to add some regularity requirements on the coefficients of the equation. 
With different techniques, the $(n-2)$-dimensional Hausdorff measure result has already 
been proved in \cite[Theorem 1.1]{hanhardtlin}; see also \cite[Theorem 7.2.1]{hanlin}.

\begin{remark}
 As noted in \cite[Remark at page 362]{hoste}, it is not possible to get effective bounds on the 
critical sets of solutions to \eqref{eq_Lu2} with \eqref{e:coefficient_estimates}. Indeed, every 
closed subset of $\R^n$ can be the critical set of such a function.
\end{remark}

\bibliographystyle{aomalpha}
\bibliography{ChNaVa_bib}

\end{document}